\documentclass[11pt]{amsart}
\usepackage{amsmath}
\usepackage{cases}
\usepackage{mathrsfs}
\usepackage{amssymb}
\usepackage{txfonts}
\usepackage{amscd}
\usepackage{braket}
\usepackage{amsfonts,latexsym,amsmath,amsthm,amsxtra,mathdots}
\usepackage[bookmarks,colorlinks]{hyperref}
\usepackage[all,cmtip]{xy}
\RequirePackage{amsmath} \RequirePackage{amssymb}
\usepackage{color}
\usepackage{mathtools}
\mathtoolsset{showonlyrefs}

     \renewcommand{\Im}{{\mathfrak{Im}\,}}

    \newcommand{\fa}{{\mathfrak{a}}} \newcommand{\fb}{{\mathfrak{b}}}

    \renewcommand{\Re}{{\mathfrak{Re}\,}}

\def\-{^{-1}}

\def\'{^{\prime}}
\def\pprime{^{\prime\prime}}

\def\-{^{-1}}

\newcommand{\delete}[1]{}

    \theoremstyle{plain}

\usepackage{colordvi}
\usepackage{multicol}
\usepackage{verbatim}
\numberwithin{equation}{section}
\newtheorem{thm}{Theorem}[section]
\newtheorem{lem}{Lemma}[section]
\newtheorem{rem}{Remark}[section]
\newtheorem{de}{Definition}[section]
\newtheorem{pro}{Proposition}[section]

\def\({\left( }
\def\){\right) }

\begin{document}

\title{The second moment of Rankin-Selberg L-function and hybrid subconvexity bound}
\author{Zhilin Ye}
\maketitle

\begin{abstract}
Let $M,N$ be coprime square-free integers. Let $f$ be a holomorphic cusp form of level $N$ and $g$ be either a holomorphic or a Maa{\ss} form with level $M$. Using a large sieve inequality, we establish a bound of the form $\sum_{g}\left|L^{(j)}\(1/2+it,f \otimes g\)\right|^2 \ll_t M+M^{2/3-\beta}N^{4/3}$ where $\beta \approx 1/500$. As a consequence, we obtain subconvexity bounds for $L^{(j)}\(1/2+it,f \otimes g\)$ for any $N<M$ satisfying the conditions above without using amplification methods. Moreover, by the symmetry, we establish a level aspect hybrid subconvexity bound for the full range when both forms are holomorphic. 
\end{abstract}

\section{Introduction and Statement of Results}
\label{se1}
\subsection{Introduction}

For an automorphic cuspidal representation $\pi$ with conductor $\mathfrak{Q}$, the generalized Lindel\"of Hypothesis states that $L\(\frac{1}{2},\pi\) \ll_{\epsilon} \mathfrak{Q}^{\epsilon}.$ This bound follows from the generalized Riemann Hypothesis. In many cases, however, the best known bound is the convexity bound
$L\(\frac{1}{2},\pi\) \ll_{\epsilon} \mathfrak{Q}^{\frac{1}{4}+\epsilon}$
which is a consequence of the Phragmen-Lindel\"of convexity principle, the functional equation for $L\(\frac{1}{2},\pi\)$, and Molteni's bound for $L\(\frac{1}{2},\pi\)$ on the line $\Re (s) =1$. \\

The subconvexity problem is to establish a bound of the form
$L\(\frac{1}{2},\pi\) \ll_{\epsilon} \mathfrak{Q}^{\frac{1}{4}-\delta}$
for some positive $\delta$.\\

When $\pi = f \otimes g$ where $f,g$ are both $GL_2$ Hecke cusp forms ($L$-function of $\pi$ is induced by an isobaric representation of $GL_4$ in this case, see \cite{RD}), some authors have successfully established level aspect subconvexity results via the amplification method one form fixed. For example, if $f$ is a Hecke cusp form of a fixed level and $g$ is a Hecke cusp form of a varying level $M$, then various bounds of the form
 $$L\(\frac{1}{2},f \otimes g\) \ll_{f} M^{\frac{1}{2}-\delta}$$
 for some absolute positive constant $\delta$ have been shown by Kowalski-Michel-VanderKam \cite{KMV}, Michel \cite{M1}, and Harcos-Michel \cite{HM1}. Furthermore, the subconvexity bound for two independently varying forms have been established in the works of Michel-Ramakrishnan \cite{MR}, Feigon-Whitehouse \cite{FW}, Nelson \cite{N1} and Holowinsky-Templier \cite{HT1} in situations where positivity of the central $L$-values is known. In addition, Holowinsky-Munshi \cite{HM} proved a hybrid subconvexity bound when the levels of $f$ and $g$ satisfy certain conditions via a second moment estimation. There are also a lot of open questions about the subconvexity problem. The level aspect subconvexity bound for the Rankin-Selberg convolution of two $GL_2$ forms of the same level.\\ 

In this paper, we continue the study of $L$-functions of the Rankin-Selberg convolution of two $GL_2$ cusp forms via a second moment method, and give a hybrid subvonvexity bound when the two levels are coprime and square-free.

\subsection{Main Results}
All the notations and normalizations in this section can be found in Chapter \ref{se11}.

\begin{pro}
\label{thm1}
Let $f$ be a holomorphic (resp. non-exceptional Maa{\ss}) Hecke newform with weight $k>2$ (resp. spectral parameter $t_f$ and weight $0$) and level $N$. Denote by $\lambda_f$ the Hecke eigenvalues of $f$. Let $\mathcal{B}_{\kappa}(M)$ be an orthogonormal basis of holomorphic Hecke eigenforms with weight $\kappa$ and level $M$. Let $h$ be a smooth function supported on $[\frac{1}{2},\frac{5}{2}]$ which satisfies $h^{(j)} \ll Z_h^j$. Assume that $N$ is square-free and $(M,N) =1$. Then we have
\begin{align}
&\sum_{g \in \mathcal{B}_{\kappa}(M)}\omega_g^{-1}\left|\sum_{n}\frac{\lambda_f(n)\lambda_g(n)}{\sqrt{n}}h\(\frac{n}{X}\)\right|^2\\
&\ll_{\epsilon,f,\kappa} \(1+\frac{X}{MN}+\frac{X}{M^{1+\beta}}+\frac{\(1+Z_h\)^{24}N^{4/3}}{M^{1/3+\beta}}\(1+\sqrt{\frac{X}{MN}}\)\)((1+Z_h)XMN)^{\epsilon},
\end{align}
where $\beta = \frac{11}{4875}$ and $\omega_g := \frac{(4\pi)^{\kappa-1}}{\Gamma(\kappa-1)}\Braket{g,g}$.
\end{pro}

\begin{rem}
For $k=2$ and $f$ holomorphic, we can only get the bound with $N^{\frac{4}{3}}$ replaced by $N^{\frac{3}{2}}$ in the error term.
\end{rem}


\begin{thm}
\label{cro1}
Let $M, N$ be positive coprime integers with $N$ square-free. Denote by $\mathcal{B}^{*}_{\kappa}(M)$ an orthonormal basis of holomorphic cusp forms with weight $\kappa$ and level $M$. Then, for any holomorphic (resp. non-exceptional Maa{\ss}) Hecke newform $f$ with weight $k>2$ (resp. spectral parameter $t_f$ and weight $0$) and level $N$, we have
$$\sum_{g \in B^{*}_{\kappa}(M)}\left|L^{(j)}\(\frac{1}{2}+it,f \otimes g\)\right|^2 \ll_{\epsilon,f,\kappa,j} \(M+\(1+\(1+|t|\)^{24}\frac{N^{\frac{1}{3}}}{M^{\frac{1}{3}}}\)M^{1-\beta}N\)(\(1+|t|\)MN)^{\epsilon},$$
where $\beta = \frac{11}{4875}$.
\end{thm}

\begin{rem}
Neither the exponent $\(1+|t|\)^{24}$ nor the power saving $\beta$ is optimal. 
\end{rem}

\begin{thm}
\label{cro1_1}
Under the condition of Theorem \ref{cro1}, let $M$ be square-free. Then there are effective constants $\alpha, B>0$ such that
$$L^{(j)}\(\frac{1}{2}+it,f \otimes g\) \ll_{\epsilon,j ,f,\kappa} (1+|t|)^B(MN)^{\frac{1}{2}-\alpha},$$
where the implied constant does not depend on the non-archimedean conductor $N$ of $f$.
\end{thm}

\begin{rem}
By combining our results and the proofs in \cite{KMV}, $\alpha$ can be at least $1/1602$. We can expect a sharper bound by using amplification method in Section \ref{se2}. Another approach to the hybrid subconvexity problem can be found in \cite{MNV}, where the authors are able to establish subconvexity for more general cases. 
\end{rem}

The key ingredient to establish the theorems above is the following Large Sieve type inequality.

\begin{thm}
\label{pro1}
Let $Z, V, H, Q \geqslant 1$ be real numbers and let $u(v,h,q,d)$ be a smooth function supported on $[V,2V]\times[H, 2H]\times[Q,2Q]\times [D,2D]$ satisfying $||u^{(i,j,k,l)}||_{\infty} \leqslant Z^{i+j+k+l}V^{-i}H^{-j}Q^{-k}D^{-l}$ for all $i,j,k,l\geqslant 0$. Let $r,s$ be positive integers s.t. $(r,s) = 1$. Let $w$ be a positive integer such that $(w,rs) =1$. Let $a(v), b(h,d)$ be two finite sequences of complex numbers. Then
\begin{align*}
\mathfrak{S}_{\pm} := & \sum_{\substack{q\\(q,r)=1}}\sum_{d}\sum_{h}\sum_{v}\frac{1}{q}S(v\overline{r}, \pm hw, sq)a(v)b(h,d)u(v,h,q,d) \\
\ll _{\epsilon} & s\sqrt{r}\left(\frac{1+\left(\frac{\Xi}{Z}\right)^{-2\theta}}{Z+\Xi}\right)\left(Z+\Xi+ \sqrt{\frac{V}{rs}}\right)\left(Z+\Xi+ \sqrt{\frac{H}{rs}}\right)w^{\theta}\|a(v)\|_2\|B(h)\|_2(1+Z^8)(wVHZ)^{\epsilon},
\end{align*}
where $\Xi = \frac{\sqrt{VHw}}{s\sqrt{r}Q}$, $B(h) = \sum_{D\leqslant d \leqslant 2D} |b(h,d)|$ and $\theta$ is the Ramanujan bound for cusp forms on $\Gamma(rs)\backslash \mathbb{H}^2$.
\end{thm}


\begin{rem}
By Kim-Sarnak \cite{K}, $\theta$ could be as small as $7/64$ for any $r,s$. If we assume the Ramanujan Conjecture for the discrete group $\Gamma(rs)$, then $\theta$ could be $\epsilon$. Moreover, the condition $(w, rs) =1$ is necessary in our proof. This inequality is inspired by the results in \cite{DI}, \cite{B1} and \cite{P1}.
\end{rem}

\begin{rem}
This inequality is a generalization of Theorem 13 in \cite{DI} (which is the case $w=1$ and $D=2/3$). However, after directly applying the result in \cite{DI}, one can only obtain $\(Z+ \Xi+\sqrt{\frac{Hw}{sr}}\)$ in the last parenthesis, rather than $\(Z+\Xi+ \sqrt{\frac{H}{sr}}\)w^{\theta}$. Since we are going to apply this large Sieve type inequality when $w$ is very large, our improvement is crucial.
\end{rem}

\begin{rem}
There is no saving in the sum over $d$. However, we need the $d$-sum here since the weight function $u$ contains an additional variable $d$ in our application.
\end{rem}

\subsection{The Structure of this Paper}
In Section \ref{se11}, we introduce all the notations, formulae and lemmas we need. In Section \ref{sed}, we reduce Theorem \ref{cro1} to Theorem \ref{thm1}. In Section \ref{ses}, we give a sketch of the proof. In Sections \ref{se2} and \ref{weight_functions}, we follow the same lines as in previous section and prove the main theorem in detail. In Section \ref{seLS}, we give the proof of the large sieve type inequality.\\

\section{Preliminaries}
\label{se11}
\subsection{Automorphic Forms and Nomalizations}\label{se1.1}
Let $M>0$ be an integer and $k>0$ be an even integer. Let 
$$\Gamma_0(M) := \left\{\left(\begin{array}{cc}a & b \\c & d\end{array}\right) \in SL_2(\mathbb{Z}): c \equiv 0 (\text{mod }M)\right\}$$
be the congruence subgroup. 

Denote by $\mathcal{L}^2(M)$ and $\mathcal{L}^2_0(M) \subset \mathcal{L}^2(M)$, respectively, the space of weight zero Maa{\ss} forms, and the space of weight zero Maa{\ss} cusp forms with respect to the congruence subgroup $\Gamma_0(M)$.
As the notations in \cite{BH}, denote by $\mathcal{S}_{k}(M)$ the linear space of all the functions $f(z) = y^{\frac{k}{2}}F(z)$ where $F$ is a holomorphic cusp form with weight $k$, level $M$ and trivial nebentypus. Both $\mathcal{L}^2_0(M)$ and $\mathcal{S}_{k}(M)$ are Hilbert spaces respect to the inner product
$$\braket{f_1,f_2} : = \int_{\Gamma(M)\backslash\mathbb{H}^2}f_1(z)\overline{f_2(z)}\frac{dxdy}{y^2}.$$

We recall the definition of holomorphic cusp forms and Maa{\ss} forms.

The holomorphic cusp forms with weight $k$ and level $M$ are holomorphic functions on the upper half-plane $F : \mathbb{H}^2 \rightarrow \mathbb{C}$ satisfying
$$F(\gamma z) = (cz+d)^{k}F(z),$$
when 
$$\gamma = \left(\begin{array}{cc}a & b \\c & d\end{array}\right) \in \Gamma_0(M),$$
 and vanishing at every cusp. Any form $f\in \mathcal{S}_{k}(M)$ has a Fourier series expansion at infinity
 $$f(z) = \sum_{n\geqslant 1}\frac{\psi_f(n)}{n^{\frac{1}{2}}}(yn)^{\frac{k}{2}}e(nz)$$
 with coefficients $\psi_f(n)$ satisfying
 $$\psi_f(n) \ll_f \tau(n)$$
as proven by Deligne. In this paper, $e(z)$ always means $e^{2\pi i z}$.

The Maa{\ss} cusp forms with archimedean parameter $\lambda_{\infty}$ and level $M$ are $L^2$ functions $f : \mathbb{H}^2 \rightarrow \mathbb{C}$ satisfying $\Delta f = \lambda_{\infty} f$ for Laplacian $\Delta = -y^2(\partial^2_x+\partial^2_y)$, $f(z) = f(\gamma z)$ for all $\gamma \in \Gamma_0(M)$ and vanishing at every cusp. Any form $f \in \mathcal{L}^2_0(M)$ has the Fourier series expansion as 
$$f(z) = \sum_{n \neq 0} \frac{\psi_f(n)}{|n|^{\frac{1}{2}}}(|n|y)^{\frac{1}{2}}K_{it}(2\pi |n|y).$$

We can choose orthonormal basis $\mathcal{B}_k(M)$ and $\mathcal{B}(M)$, respectively, of $\mathcal{S}_k(M)$ and $\mathcal{L}^2_0(M)$ which consist of eigenfunctions of all the Hecke operators $T_m$ with $(m,M) =1$. If a cusp form $f$ is an eigenfunction of the Hecke operator $T_m$, we denote by $\lambda_f(m)$ the eigenvalue of $f$.

By the property of Hecke operators, one has that 
\begin{align}\label{Heckerel}
\psi_f(m)\lambda_f(n)=\sum_{d|(m,n)}\psi_f\(\frac{mn}{d^2}\)
\end{align}
for any $m,n\geqslant 1$ with $(n, M)=1$ when $f$ is in $\mathcal{B}_k(M)$ or $\mathcal{B}(M)$. In particular, $\psi_f(1)\lambda_f(n) = \psi_f(n)$ when $(n,M)=1$.

There are subsets $\mathcal{B}^{\star}_k(M)$ and $\mathcal{B}^{\star}(M)$, respectively, of $ \mathcal{B}_k(M)$ and $\mathcal{B}(M)$ which consist of all the \emph{newforms}. It is well known that newforms are the eigenfunctions of all the Hecke operators $T_m$ even for $(m,M)\neq1$. 

In order to treat both cases simultaneously, we rewrite the Fourier expansion as 
\begin{equation}
f(z) = \sum_{n \neq 0}\frac{\psi(n)}{\sqrt{n}}W_f(|n|y)e(nx)
\end{equation}
where
\begin{equation}
W_f(y) = 
\begin{cases}
\Gamma(k)^{-\frac{1}{2}}(4\pi y)^{\frac{k}{2}}e^{-2\pi y}, & f \in \mathcal{B}_k(M),\\
y^{\frac{1}{2}}\cosh(\frac{1}{2}\pi t_f)K_{it_f}(2\pi y), & f \in \mathcal{B}(M).
\end{cases}
\end{equation}
The $t_f$ is the archimedean parameter of $f$ which is defined as 
\begin{equation}
t_f = 
\begin{cases}
\frac{k-1}{2}, & f \in \mathcal{S}_k(M),\\
\sqrt{\lambda_{\infty} - \frac{1}{4}}, & f \in \mathcal{L}^2_0(M)\text{ satisfying } \Delta f = \lambda_{\infty} f.\\
\end{cases}
\end{equation}

We call the cusp forms $f $ with real $t_f$ the non-exceptional forms and the $f$ with $t_f = ir$ for some $0<r<\frac{1}{2}$ the exceptional Maa{\ss} forms.

Now for any cusp form $f$, we normalize it such that $\psi_f(1) = \lambda_f(1)= 1$. Moreover, when $N$ is square-free, by local calculation (see \cite{GH}), we have that 
\begin{align}\label{saving_norm}
\left|\lambda_f(L)\right|=L^{-1/2} \text{ for any } L|N.
\end{align}

Under this normalization, we have
\begin{pro}\text{(Theorem 1.1 \cite{HT})}
Let $f \in \mathcal{B}(N)$ be a Hecke-Maass cusipidal newform of square-free level $N$ as normalized above. Then for any $\epsilon >0$ we have a bound
$$\|f\|_{\infty} \ll_{t_f,\epsilon} N^{\frac{1}{3}+\epsilon},$$
where the implied constant depends continuously on $\lambda$.
\end{pro}

However, in our application, we need the bound for holomorphic case too.
\begin{pro}\text{(Sup-norm for holomorphic case)}\label{sup}
Let $f \in \mathcal{B}_k(N)$ with square-free level $N$ and weight $k>2$. Then for any $\epsilon >0$ we have a bound
$$\|f(z)\|_{\infty} \ll_{k,\epsilon} N^{\frac{1}{3}+\epsilon}.$$
\end{pro}

\begin{rem}
This result is first claimed in \cite{HT}. But the author is not aware of any written proof. A complete proof can be found in \cite{Y1}.
\end{rem}

From these propositions, one can deduce the Wilton's bound as below.
\begin{lem}\label{Wilton}
Let $f$ be an element in $\mathcal{B}^{\star}_k(N)$ or $\mathcal{B}^{\star}(N)$ with square-free level $N$. Then 
\begin{align}
S(f,X,\alpha):=\sum_{n \leq X}\frac{\psi_f(n)e(n\alpha)}{\sqrt{n}}\ll \|f\|_{\infty}(NX)^{\epsilon} \ll N^{\frac{1}{3}}(NX)^{\epsilon}.
\end{align}
\end{lem}

\begin{proof}
When $f \in \mathcal{B}^{\star}_k(N)$, as in \cite{HM1} Section 2.6, we have
\begin{align}\label{shifted_sum}
\sum_{1\leqslant n\leqslant X}\frac{\psi_f(n)e(n\alpha)}{n^{\frac{1}{2}+\epsilon}} = \frac{1}{\Gamma\(\frac{k}{2}+\epsilon\)}\int_0^{1}\frac{e(-Xt)-1}{1-e(t)}\int_0^{\infty}y^{-1+\epsilon}f(t+\alpha+iy)dydt.
\end{align}

By \cite{RO} Lemma 3.1.2, we have 
\begin{align}
f(t+iy) \ll_{k,\epsilon} N^{\epsilon}y^{-\frac{1}{2}}.
\end{align}

By Proposition \ref{sup} and the bound above, we have
$$
\int_0^{\infty}y^{-1+\epsilon}f(t+\alpha+iy)dy \ll_{k,\epsilon} \int_0^{1}y^{-1+\epsilon} \|f\|_{\infty} dy + \int_1^{\infty}y^{-\frac{3}{2}+\epsilon} N^{\epsilon}dy \ll_{k,\epsilon} \|f\|_{\infty} + N^{\epsilon}.
$$

Using this upper bound in \eqref{shifted_sum}, we get
$$\sum_{1\leqslant n\leqslant X}\frac{\psi_f(n)e(n\alpha)}{n^{\frac{1}{2}+\epsilon}} \ll \|f\|_{\infty}(NX)^{\epsilon} \ll N^{\frac{1}{3}}(NX)^{\epsilon}.$$

By partial summation, we complete the proof in this case.

When $f \in \mathcal{B}^{\star}(N)$, as in the proof of \cite{HM1} Proposition 2.4,
\begin{align}\label{shifted_sum2}
& \sum_{1\leqslant n\leqslant X}\frac{\psi_f(n)e(n\alpha)}{n^{\frac{1}{2}+\epsilon}} \\
& =\frac{\pi^{\frac{1}{2}+\epsilon}}{4\Phi(\frac{1}{2}+\epsilon, it_f)}\int_0^{1}\frac{e(-Xt)-1}{1-e(t)}\int_0^{\infty}y^{-1+\epsilon}\(f(t+\alpha+iy)\pm f(-\alpha - t + iy)\)dydt.
\end{align}
where $\Phi(\frac{1}{2}+\epsilon, it_f)$ does not depend on $N$. The rest of the proof follows the same line.
\end{proof}

Also we need Rankin's result
\begin{lem}\label{Rank}
Let $f \in  \mathcal{B}_k(N)$ or $\mathcal{B}(N)$, then 
\begin{align}
\sum_{n \leq x}\frac{|\lambda_f(n)|^2}{n}\ll_{\epsilon, t_f} (N)^{\epsilon}.
\end{align}
\end{lem}


\subsection{Rankin-Selberg Convolution and L-functions}\label{se1.2}
Let $N, M$ be two positive integers and $\kappa, k$ be two fixed positive even integers. 
Given two newforms $f\in\mathcal{B}^{\star}_{\kappa}(N) \bigcup \mathcal{B}^{\star}(N)$ and $g\in\mathcal{B}^{\star}_{k}(M) \bigcup \mathcal{B}^{\star}(M)$, we consider the associated L-function
$$L(s,f\otimes g) : = \prod_{p}\prod_{i=1}^2\prod_{j=1}^2\(1-\frac{\alpha_{f,i}(p)\alpha_{g,j}(p)}{p^s}\)^{-1} = \zeta^{(NM)}(2s)\sum_{n \geqslant 1}\lambda_f(n)\lambda_g(n)n^{-s}$$
where $\{\alpha_{f,i}\}$ and  $\{\alpha_{g,j}\}$ are local parameters of the $L$-function associated to $f$ and $g$ respectively and 
$$\zeta^{(NM)}(2s) = \prod_{p \nmid NM}\(1-\frac{1}{p^{2s}}\)^{-1}.$$

The complete $L$-function is given as
$$\Lambda(s) : =  \mathfrak{Q}^{\frac{s}{2}}L_{\infty}(s,f\otimes g) L(s, f\otimes g),$$
where the conductor $\mathfrak{Q} := \mathfrak{Q}(f\otimes g)$ and the local factor at infinity is defined as a product of gamma factors 
$$L_{\infty}(s,f\otimes g) : = \prod_{i=1}^{4}\Gamma_\mathbb{R}\(s+\mu_{f\times g, i}(\infty)\),\, \ \Gamma_{\mathbb{R}}(s):=\pi^{-s/2}\Gamma(s/2),$$
where $\mu_{f \times g}$ are the Rankin-Selberg archimedean parameter which only depend on $t_f$ and $t_g$. See \cite{HM1} section 3 for further references.

We also have the functional equation
$$\Lambda\(s,f\otimes g\) = \varepsilon\(f\otimes g\)\Lambda\(1-s, \overline{f}\otimes \overline{g}\),$$
where $\varepsilon\(f\otimes g\)$ is the $\varepsilon$-factor with norm $1$.

By the local Langlands correspondence, one can verify that 
$$(MN)^2/(M,N)^4 \leqslant \mathfrak{Q}(f\otimes g) \leqslant (MN)^2/(M,N)$$

Based on the functional equation, one can obtain the following equation as in \cite{IK} section 5.2.
\begin{lem}[\textbf{Approximate Functional Equation}] \label{FE}
Let $f,g$ be holomorphic newforms, all the notations are as above. Then
\begin{align}
L\(s,f\otimes g\)=\sum_{n=1}^{\infty}\frac{\lambda_f(n)\lambda_g(n)}{n^{s}}V_s\(\frac{n}{\sqrt{\mathfrak{Q}}}\) + \varepsilon\(f\otimes g, s\)\sum_{n=1}^{\infty}\frac{\overline{\lambda_f(n)\lambda_g(n)}}{n^{1-s}}V_{1-s}\(\frac{n}{\sqrt{\mathfrak{Q}}}\)
\end{align}
where
$$V_s(y) = \frac{1}{2\pi i}\int_{(3)}G(u)\frac{L_{\infty}(s+u, f\otimes g)}{L_{\infty}(s, f\otimes g)}\zeta^{(NM)}(2s+2u)y^{-u}\frac{du}{u}$$
and
$$\varepsilon\(f\otimes g, s\) = \varepsilon\(f\otimes g\)\mathfrak{Q}^{\frac{1}{2}-s}\frac{L_{\infty}(1-s, f\otimes g)}{L_{\infty}(s, f\otimes g)}.$$
\end{lem}

One can choose 
$$G(u) = \(\cos \frac{\pi u}{4A_0}\)^{-16A_0}$$
for any positive integer $A_0$. Then when $\mathbf{Re}\(s\)\geqslant\epsilon >0$ the test function $V_s(y)$ satisfies
\begin{align}\label{Vs}
y^jV_s^{(j)}(y) \ll_{t_f,t_g,A_0,\epsilon} \mathfrak{Q}^{\epsilon}\(1+\frac{y}{\sqrt{q_{\infty}(s)}}\)^{-A_0}
\end{align}
for any $\epsilon > 0$, where $q_{\infty}(s) = \prod_{i=1}^{4}\(|s|+|\mu_{f\times g, i}(\infty)| + 3\)^{1/2}$ is the analytic conductor, see \cite{M1}.

\subsection{Properties of Bessel Functions }\label{se1.3}

Let $v$ be a complex number with $\Re(v)>-\frac{1}{2}$. Let $J_v(x)$ be Bessel function of the first kind. It could be defined through Taylor series
\begin{align}\label{TaylorJ}
J_{v}(x) = \sum_{m=0}^{\infty}\frac{(-1)^{m}}{m!\Gamma(m+v+1)}\(\frac{x}{2}\)^{2m+v}.
\end{align}
When $x\leqslant 10$, according to Taylor expansion
\begin{align}\label{JT}
x^iJ_v^{(i)}(x)\ll x^{\Re(v)},
\end{align}
where the implied constant depends on $v$ and $i$.

Let $K_v(x)$ be the $K$-Bessel function which can be written as ( see \cite{Wa} p. 206 )
\begin{equation}\label{K_integral}
K_v(x) = \(\frac{\pi}{2x}\)^{\frac{1}{2}}\frac{e^{-x}}{\Gamma(v+\frac{1}{2})}\int_{0}^{\infty}e^{-u}u^{v-\frac{1}{2}}\(1+\frac{u}{2x}\)^{v-\frac{1}{2}}du.
\end{equation}

Based on the well-known formula
$$\pi iJ_v(x) = e^{-v\pi i/2}K_v(ze^{-\pi i /2})-e^{v\pi i/2}K_v(ze^{\pi i /2})$$
for real number $z$, one can write the $J$-Bessel functions as 
\begin{align}
J_v(x) = e^{ix}W_v(x)+e^{-ix}\overline{W}_v(x)\label{JB},
\end{align}
where
\begin{align}
W_v(x)=\frac{e^{i(\frac{\pi}{2}v-\frac{\pi}{4})}}{\Gamma(v+\frac{1}{2})}\sqrt{\frac{2}{\pi x}}\int_{0}^{\infty}e^{-y}\(y\(1+\frac{iy}{2x}\)\)^{v-\frac{1}{2}}dy.
\end{align}
Moreover, since $\Re(v)>-\frac{1}{2}$, the integral is absolutely convergent. One has that
\begin{align}
x^jW_{v}^{(j)}(x) \ll \frac{1}{(1+x)^{1/2}}\label{Wk}.
\end{align}
for $x>1$.

The following lemma is similar to the one in \cite{HM}.
\begin{lem} \label{lem7}\label{lem8}
Let $k, \kappa \geqslant 2$ be fixed integers. Let $a,b,x,y >0$. Define
$$I(x,y):=\int_{0}^{\infty}h(\xi)J_{k-1}\(4\pi a\sqrt{x\xi}\)J_{\kappa-1}\(4\pi b\sqrt{y\xi}\)d\xi$$
where $h$ is a smooth function compactly supported on $[1/2, 5/2]$ such that $h^{(j)} \ll Z^j$ for some $Z>0$. Then, when $\left|a\sqrt{x}-b\sqrt{y}\right| > 2$, we have
$$I(x,y) \ll_j \(1+Z\)^j\left|a\sqrt{x}-b\sqrt{y}\right|^{-j},$$
 for any $j \geqslant 0$. Moreover, 
$$x^iy^j\frac{\partial^i}{\partial x^i}\frac{\partial^j}{\partial y^j}I(x,y) \ll_{i,j}\frac{1}{(1+a\sqrt{x})^{1/2}(1+b\sqrt{y})^{1/2}}(1+a\sqrt{x})^i(1+b\sqrt{y})^j.$$
\end{lem}

\begin{proof} 
Change of variables, $\xi = \omega^2$, gives
$$I(x,y):=2\int_{0}^{\infty}\omega h(\omega^2)J_{k-1}\(4\pi a\sqrt{x}\omega\)J_{\kappa-1}\(4\pi b\sqrt{y}\omega\)d\omega.$$ 
When $a\sqrt{x}, b\sqrt{y}>1$, use \eqref{JB}, $I(x,y)$ may be written as a sum of four similar terms, one of them being
$$\int_{0}^{\infty}e\(2\omega\(a\sqrt{x}-b\sqrt{y}\)\)\omega h(\omega^2)W_{k-1}\(4\pi a\sqrt{x}\omega\)\overline{W}_{\kappa-1}\(4\pi b\sqrt{y}\omega\)d\omega.$$ 
Repeated integration by parts gives the first statement. When $a\sqrt{x} \leqslant 1$ ( resp. $b\sqrt{y} \leqslant 1$ ), use the Taylor expansion \eqref{JT} of $J_{k-1}$ ( resp. $J_{\kappa -1}$ ), one can still get the desired bound. For the second statement, differentiate $J$-Bessel function then use the bound either for $W_{k-1}$ or $J_{k-1}$ case by case as above. 
\end{proof}

In order to analyze the Maa{\ss} form case, we also need s similar bound as following.
\begin{lem} \label{lem7_1}\label{lem8_1}
Let $\kappa \geqslant 2$ be a fixed integer and let $t$ be a fixed real number or a fixed pure imaginary number such that $t = ir$ with $0<r<1/4$. Let $a,b,x,y >0$ such that $a\sqrt{x} > 10$. Define
\begin{align}
&I_0(x,y):=\int_{0}^{\infty}h(\xi)\(Y_{2it}(4\pi a\sqrt{x\xi})+Y_{-2it}(4\pi a\sqrt{x\xi})\)J_{\kappa-1}\(4\pi b\sqrt{y\xi}\)d\xi,\\
&I_1(x,y):=\int_{0}^{\infty}h(\xi)K_{2it}(4\pi a\sqrt{x\xi})J_{\kappa-1}\(4\pi b\sqrt{y\xi}\)d\xi,
\end{align}
where $h$ is a smooth function compactly supported on $[1/2, 5/2]$ such that $h^{(j)} \ll Z^j$ for some $Z>0$. Then, 
\begin{align}
I_0(x,y) \ll_j \(1+Z\)^j\left|a\sqrt{x}-b\sqrt{y}\right|^{-j},\, \ I_1(x,y)\ll_j e^{-2\pi a\sqrt{x}}.
\end{align}
 for any $j \geqslant 0$. Moreover, 
$$x^iy^j\frac{\partial^i}{\partial x^i}\frac{\partial^j}{\partial y^j}I_{\iota}(x,y) \ll_{i,j}\frac{1}{(1+a\sqrt{x})^{1/2}(1+b\sqrt{y})^{1/2}}(1+a\sqrt{x})^i(1+b\sqrt{y})^j,$$
where $\iota = 0$ or $1$.
\end{lem}

\begin{proof}
By \eqref{K_integral} and a well-known formula 
$$-\pi Y_v(x) = e^{-v\pi i /2}K_v(xe^{-\pi i /2})+e^{v\pi i/2}K_v(xe^{\pi i/2}),$$
we have a similar decomposition as the one in \eqref{JB}. So the bound of $I_0$ follows the same lines.

By the formula $K_{v}\'(x) = -\frac{1}{2}\(K_{v-1}(x)+K_{r+1}(x)\)$ fact that $K_v(x) = K_{-v}(x)$ and \eqref{K_integral}, we can prove the bound of $I_1$.
\end{proof}

Now, in Section \ref{se2}, we will need a lemma as following:
\begin{lem}
\label{lem12}
Let $H(x)$ be a function supported on $[1/2X, 5/2X]$. Let $\kappa \geq 2$ be an integer. Let $X\geq 10\kappa$ such that $H^{i}(x)\ll \(\frac{Z}{X}\)^{i}$ for any $i$. Let $l$ be a positive integer. Let $t$ be a positive number or a purely imaginary number such that $t=ir$, where $0<r<1/2$. Then for any positive $A$, 
\begin{equation}
\left|\int_{0}^{\infty}H(x)J_l(x)J_{\kappa}(x)\frac{dx}{x}\right| \ll_{A,\epsilon}
\begin{cases}
(Z+1)\log|X|/X & \\
\(\frac{(Z+1)X}{l^2}\)^AX & \text{if }l>2k.
\end{cases}
\label{5.1}
\end{equation}
\begin{equation}
\left|\frac{\pi}{\sinh \pi t}\int_{0}^{\infty}\frac{J_{2it}(x)-J_{-2it}(x)}{2i}J_{\kappa}(x)H(x)\frac{dx}{x} \right| \ll_{A, \epsilon}
\begin{cases}
(Z+1)\log|X|/X & \text{if } t>0 \text{ or } t=ir,  \\
\(\frac{(Z+1)X}{t^2}\)^AX & \text{if } t >1.
\end{cases}
\label{5.2}
\end{equation}
\end{lem}

\begin{proof}
The first case in \eqref{5.1} comes from the integral representation of Bessel function
$$J_l(x) = \frac{1}{2\pi}\int_0^{2\pi}e^{i(x\sin u - lu)}du,$$
$$J_k(x) = e^{ix}W_k(x) + e^{-ix}\overline{W_k(x)}.$$

By integration by parts and \eqref{Wk}, we obtain
\begin{align}
\int_0^{\infty}e^{i x\sin u}\frac{H(x)}{x}J_k(x)dx & \ll \left|\frac{i}{\sin u \pm 1}\int_0^{\infty}e^{i x(\sin u\pm 1) }\(\frac{H(x)W_k^{\pm}(x)}{x}\)\'dx\right|\\
& \ll (Z+1)X^{-3/2}|\sin u \pm 1|^{-1},
\end{align}
where $W_k^{+} := W_k$ and $W_k^{-}:=\overline{W_k}$.
Otherwise, we have the trivial bound 
$$\int_0^{\infty}e^{i x\sin u}\frac{H(x)}{x}J_k(x)dx \ll X^{-1/2}.$$

Thus,
\begin{align*}
\int_{0}^{\infty}H(x)J_k(x)J_{l}(x)dx & = \frac{1}{2\pi}\int_0^{2\pi}e^{-ilu}\int_0^{\infty}e^{ix\sin u}\frac{H(x)}{x}J_k(x)dxdu\\
& \ll \int_{0}^{2\pi}X^{-1/2}\min\{1, (Z+1)X^{-1}|\sin u \pm 1|^{-1}\}du \ll \frac{(Z+1)\log |X|}{X}.
\end{align*}

The proof of the first case in \eqref{5.2} follows the same lines. We use the integral representation
$$\frac{J_{2it}(x)-J_{-2it}(x)}{\sinh\pi t} = \frac{4i}{\pi}\int_0^{\infty}\cos(x\cosh u)\cos (2tu)du.$$
Then, when $t$ is real,
\begin{align*}
& \frac{\pi}{\sinh \pi t}\int_{0}^{\infty}\frac{J_{2it}(x)-J_{-2it}(x)}{2i}H(x)J_k( x)\frac{dx}{x}\\
& =  2\int_0^{\infty}\cos (2tu)\int_0^{\infty}\cos(x\cosh u)J_k(x)\frac{H(x)}{x}dxdu\\
& \ll \int_0^{\infty}X^{-1/2}\min\{1, (Z+1)X^{-1}|\cosh u \pm 1|^{-1}\}du \ll \frac{(Z+1)\log|X|}{X}.
\end{align*}
When $t=ir$, since $0<r < 1/2$, the second last inequality becomes
\begin{align*}
&2\int_0^{\infty}\cos (2iru)\int_0^{\infty}\cos(x\cosh u)J_k(x)\frac{H(x)}{x}dxdu\\
 \ll &\int_0^{\infty}X^{-1/2}e^{2ru}\min\{1, (Z+1)X^{-1}|\cosh u \pm 1|^{-1}\}du \ll \frac{(Z+1)\log|X|}{X}.
\end{align*}

For the second case in \eqref{5.2}, we need the differential equation of J-Bessel function $J_{\alpha}(x)$ such that
$$x^2\frac{d^2y}{dx^2}+x\frac{dy}{dx}+(x^2-\alpha^2)y =0.$$

Like (2.15) in \cite{BHM}, one can check that for any compactly supported function $f(x)$ in $C^{\infty}\((0,\infty)\)$, we have
\begin{align*}
\int_0^{\infty}f(x)J_{\alpha}(x)dx = \int_0^{\infty}\(-\(\frac{x^2f(x)}{x^2-\alpha^2}\)^{\prime\prime}+ \(\frac{xf(x)}{x^2-\alpha^2}\)\'\)J_{\alpha}(x)dx.
\end{align*}

Let $\varphi(x)$ be a compact supported smooth function. Applying the formula above to $(x^2-l^2)\varphi(x)J_k(x)$ and the differential equation of $J_k(x)$, we obtain
\begin{align*}
\int_{0}^{\infty}\varphi(x)J_k(x)J_l(x)dx = & \frac{1}{k^2-l^2}\int_{0}^{\infty}D_{11}\(\varphi\)J_k(x)\'J_l(x)dx  +\frac{1}{k^2-l^2}\int_{0}^{\infty}D_{12}\(\varphi\)J_k(x)J_l(x)dx,
\end{align*}
where 
$$D_{11}\(\varphi\) : = 2\left[x\varphi-\(x^2\varphi\)\'\right], \, \ D_{12}\(\varphi\) : = \(x\varphi\)\'-\(x^2\varphi\)^{\prime\prime}.$$

By a similar argument, we have
\begin{align}
\int_{0}^{\infty}\varphi(x)J_k(x)\'J_l(x)dx = & \frac{1}{k^2-l^2}\int_{0}^{\infty}D_{21}\(\varphi\)J_k(x)\'J_l(x)dx + \frac{1}{k^2-l^2}\int_{0}^{\infty}D_{22}\(\varphi\)J_k(x)J_l(x)dx,
\end{align}
where
$$D_{21}\(\varphi\) : = -(x\varphi\'+x^2\varphi\pprime),\, \ D_{22}\(\varphi\): = 2(x\varphi+(x^2-k^2)\varphi\').$$

Finally, let $\varphi = H$, which has type $(1: Z)$ (see section \ref{weight_functions} for the definition) and the support of size $X \geqslant 10k$. Therefore, $D_{ij}\(\varphi\)$ has type $(X(Z+1): Z+1)$ by the lemmas in section \ref{weight_functions}. By repeating using the above two formulae $A$ times (with $\varphi = D_{ij}$), we obtain the bound for the second case in (\ref{5.1}). (\ref{5.2}) follows the same lines.\\
\end{proof}


\subsection{Voronoi Formula, Trace Formulae and Large Sieve Inequalities}\label{se1.4}

We recall the Voronoi summation formulae [\cite{KMV} Appendix A3 and A.4]. Let $q>0$ be an integer, and write $N_2 : = N/(N,q)$. For a cusp form $f$ and $y>0$ define $\mathcal{J}^{\pm} = \mathcal{J}^{\pm}_f$ as following
\begin{equation}\label{Voronoi_weight_h}
\mathcal{J}^{+}(y) = 2\pi J_{k-1}(4\pi y), \, \ \mathcal{J}^{-}(y)=0
\end{equation}
if $f$ is holomorphic of weight $k$, and
\begin{equation}\label{Voronoi_weight_h}
\mathcal{J}^{+}(y) = \frac{\pi}{\cosh(\pi t)}\(Y_{2it}(4\pi y)+Y_{-2it}(4\pi y)\), \, \ \mathcal{J}^{-}(y)=4\cosh(\pi t)K_{2it}(4\pi y)
\end{equation}
if $f$ is Maa{\ss} and $t=t_f$ is the archimedean parameter defined before. Our notations are as the same as the ones in \cite{BH} Section 3.

\begin{rem}
The Bessel function $\mathcal{J}$ in \cite{BH} equals the one defined in \cite{KMV}. Recall the equality
$$\frac{\pi}{2\sin(-v/4)}\(J_v(x)-J_{-v}(x)\) = \frac{-\pi}{2\cos(v/4)}\(Y_v(x)+Y_{-v}(x)\) = K_v(-ix)+K_{v}(ix).$$
Our $t$ is as the same as the $t$ in \cite{BH}, which is the $r$ in \cite{KMV}. The reason we use $Y$-Bessel instead of $J$-Bessel is that, even when $t$=0, our $\mathcal{J}$ is still well defined. 
\end{rem}

\begin{rem}
\label{General_J_Bessel}
We list the bounds for $\mathcal{J}^{\pm}_f(y)$ for a fixed cusp form $f$. When $f$ is holomorphic $\mathcal{J}^{+}_f(y) \ll y^{k-1}$ when $y \rightarrow 0^+$ and $\mathcal{J}^{+}_f(y) \ll y^{-1/2}$ when $y \rightarrow \infty$. When $f$ is Maa{\ss}, let $r_f = -i\Im t_f \geqslant 0$. Then $\mathcal{J}^{+}_f(y) \ll y^{-2r_f}$ when $y \rightarrow 0^+$ and $\mathcal{J}^{+}_f(y) \ll y^{-1/2}$ when $y \rightarrow \infty$. If $r_f =0$, $\mathcal{J}^{-}_f(y) \ll y^{-\epsilon}$ when $y \rightarrow 0^+$. Else, if $r_f >0$, $\mathcal{J}^{-}_f(y) \ll y^{-2r_f}$ when $y \rightarrow 0^+$. We always have that $\mathcal{J}^{-}_f(y) \ll e^{-y}$ when $y \rightarrow \infty$.
\end{rem}

\begin{lem}\label{lemV}
(\textbf{Voronoi Summation Formula }) Let $(a,q) =1$ and let $h$ be a smooth function, compactly supported on $(0,\infty)$. Let $f$ be a newform of a square-free level $N$ . Set $N_2 := N/(N,q)$. Then there exists a complex number $\eta^{\pm}$ both with norm $1$ and a newform $f^{*}$ of the same level $N$  such that
$$\sum_{n}\frac{\lambda_f(n)}{\sqrt{n}}e\(n\frac{a}{q}\)h(n) = 2\sum_{\pm}\eta^{\pm} \sum_{n}\frac{\lambda_{f^*}(n)}{\sqrt{n}}e\(\mp n\frac{\overline{aN_2}}{q}\)\int_0^{\infty}h\(\frac{\xi^2q^2N_2}{n}\)\mathcal{J}^{\pm}(\xi)d\xi$$ 
where $\bar{x}$ denotes the multiplicative inverse of $x$.
\end{lem}

The explicit expression of $\eta^{\pm}$ is obtained in \cite{KMV}. When $f$ is holomorphic, $\eta^{\pm} = i^k\eta(N_2)$, where $\eta(N_2)$ is the pseudo-eigenvalue of the Atkin-Lehner operator $W_{N_2}$. When $f$ is Maa{\ss}, $\eta^{+} = \eta(N_2)$, $\eta^{-} = \epsilon_f\eta(N_2)$ where $\epsilon_f$ is the eigenvalue of $f$ under the reflection operator. Therefore, in any case, $\eta^{\pm}$ only depends on $N_2$ and $f$.


For any $m,n,c \in\mathbb{N}$, let $S(m,n; c)$ denote the Kloosterman sum
$$S(m,n;c) = \sideset{}{^*}\sum_{a(c)} e\(\frac{n\alpha+m\bar{a}}{c}\).$$

Sum of Kloosterman sums appear in trace formula. First we have the following.
\begin{lem}\textbf{(Petersson Trace Formula)}\label{lemPT}
Let $N\geqslant 1$ be an integer. Let $\mathcal{B}_{k}(M)$ be any Hecke eigenbasis for $\mathcal{S}_k(M)$. For any $n,m \geqslant 1$, we have
$$\sum_{f\in \mathcal{B}_k(M)}\omega_f^{-1}\psi_f(n)\overline{\psi_f}(m) = \delta(n,m) + 2\pi i^{-k}\sum_{\substack{c>0\\ c\equiv 0(M)}}\frac{1}{c}S(n,m;c)J_{k-1}\(\frac{4\pi\sqrt{nm}}{c}\)$$
where the spectral weights $\omega_f$ are given by
$$\omega_f := \frac{(4\pi)^{k-1}}{\Gamma(k-1)}\braket{f,f}$$
and $\delta(n,m) =1$ if $n=m$ and $\delta(n,m) = 0$ otherwise.\\
\end{lem}

Next, we state the Kuznietsov Trace formula and the bound of its weight functions. See \cite{DI} Theorem 1. For the definition of $S_{\mathfrak{ab}}(m, n ; \gamma)$, see \cite{DI} (1.6).

\begin{lem}\textbf{(Kuznietsov Trace Formula)}\label{lemKT}
Let $m,n$ be two positive integers and $\varphi$ a $C^{3}$-class function with compact support on $(0,\infty)$; let $\fa$ and $\fb$ be two cusps of $\Gamma = \Gamma_0(q)$; denoting by $\sum^{\Gamma}$ a summation performed over the positive real numbers $\gamma$ for which $S_{\mathfrak{ab}}(m, n ; \gamma)$ is defined, one has
\begin{align*}
\sum_{\gamma}^{\Gamma} \frac{1}{\gamma} S_{\mathfrak{ab}}(m, n ; \gamma) \varphi(\frac{4\pi \sqrt{mn}}{\gamma}) & = \frac{1}{2\pi}\sum_{k = 0(2)}\sum_{j}\frac{i^k(k-1)!}{(4\pi)^{k-1}}\overline{\psi}_{jk}(\mathfrak{a},m)\psi_{jk}(\mathfrak{b},n)\tilde{\varphi}(k-1)\\
& + \sum_{j \geqslant 1} \frac{\overline{\rho}_{j\mathfrak{a}}(m)\rho_{j\mathfrak{b}}(n)}{\cosh (\pi t_j)}\hat{\varphi}(t_j)\\
& + \frac{1}{\pi}\sum_{j}\int_{-\infty}^{\infty}\left(\frac{m}{n}\right)^{-it}\overline{\varphi}_{j\mathfrak{a}}\left(m,\frac{1}{2}+it\right)\varphi_{j\mathfrak{b}}\left(n,\frac{1}{2}+it\right)\hat{\varphi}(t)dt,
\end{align*}
and
\begin{align*}
\sum_{\gamma}^{\Gamma} \frac{1}{\gamma} S_{\mathfrak{ab}}(m, -n ; \gamma) \varphi(\frac{4\pi \sqrt{mn}}{\gamma}) & = \sum_{j \geqslant 1} \frac{\overline{\rho}_{j\mathfrak{a}}(m)\rho_{j\mathfrak{b}}(n)}{\cosh (\pi t_j)}\check{\varphi}(t_j)\\
& + \frac{1}{\pi}\sum_{j}\int_{-\infty}^{\infty}\left(mn\right)^{-it}\overline{\varphi}_{j\mathfrak{a}}\left(m,\frac{1}{2}+it\right)\varphi_{j\mathfrak{b}}\left(n,\frac{1}{2}+it\right)\check{\varphi}(t)dt
\end{align*}
where the Bessel transforms are defined by
$$\tilde{\varphi}(l)=\int_{0}^{\infty}J_l(y)\varphi(y)\frac{dy}{y},$$
$$\hat{\varphi}(t)=\frac{\pi}{\sinh \pi t}\int_{0}^{\infty}\frac{J_{2it}(x)-J_{-2it}(x)}{2i}\varphi(x)\frac{dx}{x},$$
$$\check{\varphi}(t) = \frac{4}{\pi}\cosh \pi t \int_{0}^{\infty}K_{2it}(x)\varphi(x)\frac{dx}{x}.$$
\end{lem}

\begin{rem}\label{SSS}
Let $\Gamma =\Gamma(rs)$, where $r,s$ are coprime integers. In our proof, we only need the fact that $S_{\infty 1/s}(m,n; \gamma) = e\(n\frac{\bar{s}}{r}\)S(m\bar{r},n;sC)$ (see \cite{DI} (1.6)), and the identity (See \cite{DI} section 1)
$$\sum^{\Gamma}_{\gamma} \frac{1}{\gamma}S_{\infty 1/s}(m,n; \gamma)= \sum_{\substack{C>0 \\ (C,r)=1}}\frac{1}{s\sqrt{r}C}e\(n\frac{\bar{s}}{r}\)S(m\bar{r},n;sC).$$
\end{rem}

\begin{rem}\label{coef}
$\rho_{j\mathfrak{a}}(\cdot), \psi_{jk}(\mathfrak{a}, \cdot)$ are the Fourier coefficients of Maass forms and holomorphic cusp forms at cusp $\mathfrak{a}$ respectively. The normalization of $\rho_j, \psi_{jk}$ in this lemma is different from the normalization we have in Sections \ref{se1.1} and \ref{se2}. With the later one, for any $f \in \mathcal{B}^{*}_k(M)$, $\psi_f(1)$ is normalized to be $1$. Thus, $\braket{f,f} = \frac{3\text{Vol}(\Gamma(M)\backslash\mathbb{H}^2)}{2\pi}\text{Res}_{s=1}L(s,f\times f) \gg M^{1-\epsilon}$ (see \cite{HL}). However, in the above lemma, $\psi_{f_j}(1)$ is normalized such that $\braket{f_j,f_j} = 1$. The reader can see \cite{DI} for more details.
\end{rem}

\begin{rem}\label{Eis}
$ \varphi_{j\mathfrak{a}}(\cdot)$ is the Fourier coefficient of Eisenstein series. The sum over $j$ is a finite sum over a "suitable" parametrization of Eisenstein series. In the classical pattern, which is also the case in \cite{DI}, the parametrization is chosen to be the set of the cusps of $\Gamma_0(q)$. In the adelic reformulation of theory of cusp forms, we have another natural basis as a finite set described in \cite{GJ}. See \cite{BHM1} for more details.
\end{rem}

Next, we need to estimate the functions on the spectral side. We quote a part of Lemma 2.1 in \cite{P1} as following.
\begin{lem}\text{(\cite{P1} Lemma 2.1)}\label{lemB}
Let $X > 0$, $Z \geqslant 1$, $k \in \mathbb{N}$. If $\varphi$ is supported on $[X, 2X]$ with derivatives of orders $v = 0,1,2,\dots 2k$ bounded by $\varphi^{(v)} \ll \(\frac{Z}{X}\)^v$ then the following bounds hold. \\
(a) For any real $t$ and any $l \leqslant k$,
\begin{equation*}
\hat{\varphi}(t), \check{\varphi}(t), \tilde{\varphi}(t) \ll \left(\frac{Z^2 + t^2}{X^2}\right)^l\left(\frac{1+|\log (X/Z)|}{1+X/Z}\right).
\end{equation*}
(b) For all real $t$ with $|t| > \max\{2X,1\}$ and for all $j \leqslant 2k$,
\begin{equation*}
\hat{\varphi}(t), \check{\varphi}(t), \tilde{\varphi}(t) \ll \left(\frac{Z}{|t|}\right)^j\left(\frac{1}{|t|^{1/2}}+\frac{X}{|t|}(1+\log |t|)\right)
\end{equation*}
(c) For exceptional eigenvalues $\lambda = 1/4 + (it)^2$ we take $t \in (0, 1/2)$ and
\begin{equation*}
\hat{\varphi}(it), \check{\varphi}(it) \ll
\begin{cases}
\frac{Z}{X} & \text{if}\ X>Z \\
\{1+ \log \frac{Z}{X}\}\left(\log \frac{Z}{X}\right)^{2t} & \text{if} X<Z.\\
\end{cases}
\end{equation*}
\end{lem}

For those coefficients appeared in the Kuznietsov Trace formula, one has the following large sieve type inequality.
\begin{lem}\text{(\cite{DI} Theorem 2)}
\label{lem13}
Let $T,K \geqslant 1$. Set $\Gamma = \Gamma(q)$. Let $\mathfrak{a} = u/w$ be a cusp of $\Gamma$ where $w|q$ and $(u,w)=1$. Let $\rho, \mu, \varphi$ be defined as in Lemma \ref{lemKT}. Then, for any sequence of complex numbers $\{a_k\}$ ,
\begin{align*}
& \sum_{|t_j| \leqslant T}\frac{1}{\cosh (\pi t_j)}\left|\sum_{k \sim K} a_k \rho_{j\mathfrak{a}}(k)\right|^2 \\
& \sum_{k = 0(2), k \leqslant T}\sum_{j}\frac{i^k(k-1)!}{(4\pi)^{k-1}}\left| \sum_{k \sim K} a_k \psi_{jk}(\mathfrak{a},k)\right|^2\\
& \sum_{j}\int_{-T}^{T}\left| \sum_{k \sim K} a_k k^{it} \varphi_{ja}\left(m,\frac{1}{2}+it\right)\right|^2dt
\end{align*}
are all bounded by $O_{\epsilon}\left(\left(T^2 + \mu(\mathfrak{a})K^{1+\epsilon}\right)||a_k||\right)$, where $\mu(\mathfrak{a}) = (w , q/w) q^{-1}$.\\
\end{lem}

\begin{lem}\label{lemLS}
Let $\eta$ be a smooth function supported on $[C/2,5C/2]$ such that $\eta^{(j)}\ll_j C^{-j}$ for all $j\geq 0$. For any sequences of complex numbers $x_n, y_m$ we have
\begin{align*}
\sum_{x \leq X}\sum_{y \leq Y}x_ny_m\sum_{\substack{c>0 \\ c \equiv 0(N)}}\frac{\eta(c)}{c}S(n,m;c)J_{k-1}\(\frac{4\pi \sqrt{mn}}{c}\)
\ll_{\epsilon,k}C^{\epsilon} \(\frac{\sqrt{XY}}{C}\)^{k-3/2}\(1+\frac{X}{N}\)^{1/2}\(1+\frac{Y}{N}\)^{1/2}\|x\|_2\|y\|_2
\end{align*}
with any $\epsilon > 0$. Moreover the exponent $k-3/2$ may be replaced by $1/2$.\\
\end{lem}

\subsection{Jutila's Circle Method}\label{se1.5}
For any collection of integers $\mathfrak{Q} \subseteq [1,Q]$, and a positive real number $\delta$ such that $Q^{-2} \ll \delta \ll Q^{-1}$, we define the function

$$\tilde{I}_{\mathfrak{Q},\delta}(x) = \frac{1}{2\delta \Lambda} \sum_{q\in \mathfrak{Q}}\sideset{}{^*}\sum_{a \textrm{ mod } q} I_{[\frac{a}{q}-\delta, \frac{a}{q}+\delta]}(x),$$
where $\Lambda = \sum_{q \in \mathfrak{Q}} \varphi(q)$ and $I_{[a,b]}(x)$ is the characteristic function of interval $[a,b]$. 
Moreover, it is an approximation of $I_{[0,1]}$ in the following sense (for a simple proof, see Lemma 4 in \cite{M2}):

\begin{lem}
\label{thm2}
We have
$$\int^{\infty}_{-\infty}\left|I_{[0,1]}(x)-\tilde{I}_{\mathfrak{Q},\delta}(x)\right|^2dx \ll \frac{Q^{2+\epsilon}}{\delta \Lambda^2}.$$\\
\end{lem}

\section{The Deduction of Theorem \ref{cro1} and Corollary \ref{cro1_1} from Proposition \ref{thm1}} \label{sed}
In this section, we prove Theorem \ref{cro1} and Corollary \ref{cro1_1} by assuming Theorem \ref{thm1}.

\begin{proof}[Proof of Theorem \ref{cro1}] 
First, we assume that $f$ (resp. $g$) is a holomorphic new Hecke eigenform on $\mathbb{H}^2$ with level $N$ (resp. $M$) and weight $\kappa$ (resp. $k$).  Also assume that $N$ is square-free, $M$ and $N$ are coprime, $N<M$. Then $\mathfrak{Q}:=\mathfrak{Q}(f\otimes g) = (NM)^2$ is the conductor of $f\otimes g$ in this case by the argument in Section \ref{se1.2}.

Assume that $s=\frac{1}{2}+\mu$ in Lemma \ref{FE}, where $\mathbf{Re}(\mu) < \frac{2}{\log \mathfrak{Q}}$. From the Approximate Functional Equation, we have 
\begin{align*}
&L\(\frac{1}{2}+\mu,f\otimes g\)=\sum_{n=1}^{\infty}\frac{\lambda_f(n)\lambda_g(n)}{\sqrt{n}}n^{-\mu}V_{\frac{1}{2}+\mu}\(\frac{n}{\sqrt{\mathfrak{Q}}}\)+\varepsilon\(f\otimes g,\frac{1}{2}+\mu\)\sum_{n=1}^{\infty}\frac{\overline{\lambda_f(n)}\overline{\lambda_g(n)}}{\sqrt{n}}n^{\mu}V_{\frac{1}{2}-\mu}\(\frac{n}{\sqrt{\mathfrak{Q}}}\).
\end{align*}

Assume that $h_0(x)$ is a positive smooth function, compactly supported on $[\frac{1}{2},\frac{5}{2}]$ with bounded derivatives. And  when $X$ runs over values $2^{v}$ with $v =  -1, 0, 1, 2, \dots$, for any $x\geqslant1$,
$$\sum_{X}h_0\(\frac{x}{X}\)=1.$$

From \eqref{Vs}, after applying a smooth partition of unity and Cauchy inequality, one can obtain that 
$$L\(\frac{1}{2}+\mu,f\otimes g\) \ll_{t_f,\kappa,A,\epsilon} \mathfrak{Q}^{\epsilon} \sum_{X} \left|L_{f\otimes g}(X)\right| \(1+\frac{X}{\(|\mu|+1\)\sqrt{\mathfrak{Q}}}\)^{-A_0}$$
where
$$L_{f\otimes g}(X) :=\sum_{n}\frac{\lambda_f(n)\lambda_g(n)}{\sqrt{n}}h_0\(\frac{n}{X}\)n^{\pm\mu}V_{\frac{1}{2}\pm\mu}\(\frac{n}{\sqrt{\mathfrak{Q}}}\)$$.

Then, for a fixed $X\gg\sqrt{\mathfrak{Q}}M^{\epsilon}$, we choose $A_0 = 100/\epsilon$ and $h\(\frac{x}{X}\) = h_0\(\frac{x}{X}\)x^{\pm\mu}V_{\frac{1}{2}\pm\mu}\(\frac{x}{\sqrt{\mathfrak{Q}}}\)$ to get 
$$L\(\frac{1}{2}+\mu,f\otimes g\) \ll_{t_f,\kappa,\epsilon}\mathfrak{Q}^{\epsilon}\sum_{X \ll \sqrt{\mathfrak{Q}}\((1+|\mu|)M\)^{\epsilon}}\left|L_{f\otimes g}(X)\right| +O_{\epsilon}\(\((1+|\mu|)M\)^{-50}\).$$

Also for such $h(x)$, one can verify that 
$$h^{(j)}(x) \ll_j (1+|\mu|)^j.$$

Next, summing over all the $g \in \mathcal{B}^*_{\kappa}(M)$ and applying Cauchy-Schwartz, we have
\begin{align*}
 \sum_{g \in \mathcal{B}^*_{\kappa}(M)} \omega_g^{-1}\left| L\(\frac{1}{2}+\mu,f\otimes g\) \right|^2
& \ll_{t_f,\kappa,\epsilon}
\mathfrak{Q}^{\epsilon}\sum_{X \ll \sqrt{\mathfrak{Q}}\((1+|\mu|)M\)^{\epsilon}}\sum_{g \in \mathcal{B}_{\kappa}(M)}\omega_g^{-1}\left| L_{f\otimes g}(X) \right|^2 +\((1+|\mu|)M\)^{-50}.
\end{align*}

Eventually, applying Theorem \ref{thm1}, we have
\begin{align*}
 \sum_{g \in \mathcal{B}^*_{\kappa}(M)} \omega_g^{-1}\left| L\(\frac{1}{2}+\mu,f\otimes g\) \right|^2 \ll
\(1+\(1+\(1+|t|\)^{15}\frac{N^{\frac{1}{3}}}{M^{\frac{1}{3}}}\)M^{-\beta}N\)(\(1+|t|\)MN)^{\epsilon}.
\end{align*}

Since $g \in \mathcal{B}^*_{\kappa}(M)$ is normalized such that $\psi_g(1) = 1$, we have that $M^{1-\epsilon}\ll\omega_g = \braket{g,g}\ll M^{1+\epsilon}$ (see \cite{HL}). Therefore, we can obtain the final bound.\\

Finally, the derivative of $L$-function at $\frac{1}{2}+it$ can be represented as an integral over the circle with radius $\frac{1}{\log \mathfrak{Q}}$ and center $\frac{1}{2}+it$. So the second moment of $L$-function itself gives the second moment of all the derivatives.
\end{proof}

\begin{proof}[Proof of Corollary \ref{cro1_1}]
By symmetry, one can assume that $N<M$. Then from \cite{KMV}, there are effective numbers $A, C>1$ and $\delta>0$ such that
 \begin{align}\label{0_1}
 L\(\frac{1}{2}+it, f\otimes g\) \ll (1+|t|)^CN^AM^{\frac{1}{2}-\delta+\epsilon},
 \end{align}
 where $A \geqslant 1$.
 
 Theorem \ref{cro1} gives that 
  \begin{align}\label{0_2}
 L\(\frac{1}{2}+it, f\otimes g\) \ll (1+|t|)^{12}(MN)^{\frac{1}{2}+\epsilon}\(N^{-\frac{1}{2}}+M^{-\frac{\beta}{2}}\).
 \end{align}
 
 Let, $N=M^{x}$. When $x \leqslant \frac{\delta}{A}$, \eqref{0_1} is bounded by $\mathfrak{Q}^{\frac{1}{4}-\frac{\delta}{4(A+\delta)}+\epsilon}$. When $\frac{\delta}{A} < x \leqslant 1$, \eqref{0_2} is bounded by $\mathfrak{Q}^{\frac{1}{4}+\epsilon}\(\mathfrak{Q}^{-\frac{\delta}{4(A+\delta)}}+\mathfrak{Q}^{-\frac{\beta A}{4(A+\delta)}}\)$. 
\end{proof}

\begin{rem}
By carefully going through the proof in \cite{KMV}, we can choose $\delta = \frac{1}{80}$ and $A=10$. So the final bound can be $\mathfrak{Q}^{\frac{1}{4}-\frac{1}{3204}+\epsilon}$. One can use amplification method in our argument to get a sharper bound. However, to keep our argument short, we simply use their results. 
\end{rem}

\section{A Sketch Proof of Propositon \ref{thm1}}\label{ses}
In this section, we provide a sketch of the proof. It follows the same lines as in \cite{HM} until step 7.

For simplicity, we assume that $M=Q, N=P$ are both primes in this section. Then $\mathfrak{Q} =(PQ)^2$ is the conductor. Let $f \in \mathcal{B}^{*}_k(P)$ be a Hecke newform. Its Fourier coefficients are normalized such that $\psi_f(n)=\lambda_f(n)$.  Furthermore, we restrict to the case of $X \sim \mathfrak{Q}^{1/2} = PQ$. Therefore, we only show the following in sketch 
$$\sum_{g \in \mathcal{B}_{\kappa}(Q)}\omega_g^{-1}\left|\sum_{n\sim PQ}\frac{\lambda_f(n)\psi_g(n)}{\sqrt{n}}\right|^2 \ll_{\epsilon,k,\kappa} P\(\frac{1}{P}+\frac{1}{Q^{\beta}}+\frac{P^{1/3}}{Q^{1/3+\beta}}\)(PQ)^{\epsilon},$$
where $n \sim PQ$ means that $n$ varies from $PQ$ to $2PQ$.\\

\textbf{Step 1. Reducing to sum of Kloosterman sums via trace formula. }
Now, after expanding the square and applying Petersson formula (Lemma \ref{lemPT}), we have
\begin{align}
&\sum_{g \in \mathcal{B}_{\kappa}(Q)}\omega_g^{-1}\left|\sum_{n \sim PQ}\frac{\lambda_f(n)\psi_g(n)}{\sqrt{n}}\right|^2  = \sum_{n \sim PQ}\frac{|\lambda_f(n)|^2}{n} + \sum_{\substack{d\equiv 0(Q)}}\sum_{\substack{n \sim PQ\\ m\sim PQ} }\frac{1}{d}S(m,n;d)\frac{\lambda_f(n)\overline{\lambda_f(m)}}{\sqrt{nm}} J_{\kappa-1}\(\frac{4\pi\sqrt{mn}}{d}\).
\end{align}

A well-known Rankin's result tells that the first term is bounded by a constant of size at most $(PQ)^{\epsilon}$.\\

\textbf{Step 2. Removing large and small values of $D$. }
Now, through the idea in section \ref{se2}, we can truncate $d$ into the range such that $d \sim PQ$ with a loss at most of size $P\(\frac{1}{P}+\frac{1}{Q^{\beta}}\)$. Next, $\sqrt{mn}/{d} \sim 1$ in this range, which is also the transition range for the Bessel function. 

Therefore, we only need to consider 
$$R_f := \sum_{\substack{d \equiv 0(Q)\\ d\sim PQ}}\sum_{\substack{n \sim PQ\\ m\sim PQ} }\frac{1}{d}S(m,n;d)\frac{\lambda_f(n)\overline{\lambda_f(m)}}{\sqrt{nm}} .$$

\textbf{Step 3. Applying the Voronoi formula to convert Kloosterman Sums to Ramanujan sums. }
Now, apply Voronoi formula (Lemma \ref{lemV}) on $n$, we have 
\begin{align*}
&R_f = \sum_{\substack{d\equiv 0(Q)\\ d\sim PQ}}\sum_{\substack{n \sim PQ\\ m\sim PQ} }\frac{1}{d}\sideset{}{^*}\sum_{a(d)}e\(\frac{\bar{a}m+an}{d}\)\frac{\lambda_{f^{*}}(n)\overline{\lambda_f(m)}}{\sqrt{nm}}\\
& \approx \sum_{\substack{d \equiv 0(QP)\\ d\sim PQ}}\sum_{\substack{n \sim PQ\\ m\sim PQ} }\frac{1}{d}\sideset{}{^*}\sum_{a(d)}e\(\frac{\bar{a}m}{d}\)e\(-\frac{\bar{a}n}{d}\)\frac{\lambda_{f^{*}}(n)\overline{\lambda_f(m)}}{\sqrt{nm}} + \sum_{\substack{\substack{d \equiv 0(Q)\\ d \nequiv 0(P)}}}\sum_{\substack{n \sim P^2Q\\ m\sim PQ} }\frac{1}{d}\sideset{}{^*}\sum_{a(d)}e\(\frac{\bar{a}m}{d}\)e\(-\frac{\overline{aP}n}{d}\)\frac{\lambda_{f^{*}}(n)\overline{\lambda_f(m)}}{\sqrt{nm}}\\
&\approx \sum_{\substack{d \equiv 0(QP)\\ d\sim PQ}}\sum_{\substack{n \equiv m (d)\\ n,m\sim PQ} }\frac{\lambda_{f^{*}}(n)\overline{\lambda_f(m)}}{\sqrt{nm}} + \sum_{\substack{\substack{d \equiv 0(Q)\\ d\sim PQ\\ d \nequiv 0(P)}}}\sum_{\substack{Pm \equiv n (d)\\n \sim P^2Q\\ m\sim PQ} }\frac{\lambda_{f^{*}}(n)\overline{\lambda_f(m)}}{\sqrt{nm}}.
\end{align*}

The first term contains only constant many terms with respect to $d$ and it is bounded by $(PQ)^{\epsilon}$.

\begin{rem}
The main difference between the trivial nebentypus case and nontrivial nebentypus case is that, after applying Voronoi formula in the later case, we can only get Gaussian sums rather than Ramanujan sums. In order to deal with the later case, we need to use trace formula reversely and apply a subconvexity bound of $GL_2 \times GL_1$ as the way in \cite{HM1}.
\end{rem}

\textbf{Step 4. Treating the Zero Shift. }
For the second term above, we consider the case that $Pm=n$. we have
\begin{align}
\sum_{\substack{\substack{d \equiv 0(Q)\\ d\sim PQ\\ d \nequiv 0(P)}}}\sum_{\substack{Pm =n\\n \sim P^2Q\\ m\sim PQ} }\frac{\lambda_{f^{*}}(n)\overline{\lambda_f(m)}}{\sqrt{nm}}\ll (PQ)^{\epsilon}.
\end{align} 
Here we used 
Rankin's bound Lemma \ref{Rank}, multiplicity of Hecke-eigenvalus, and the bound  $\left|\lambda_f(P)\right|= P^{-1/2}$.\\

\textbf{Step 5. Applying the Circle Method. }
Let $Pm - n =rd$. We are left with the case that $r\neq 0$.

Apply the circle method to detect the relation $Pm - n =rd$ for some nonzero integers $r$. We have
\begin{align}
&\sum_{\substack{\substack{d \equiv 0(Q)\\ d\sim PQ\\ d \nequiv 0(P)}}}\sum_{\substack{Pm \equiv n (d)\\n \sim P^2Q\\ m\sim PQ} }\frac{\lambda_{f^{*}}(n)\overline{\lambda_f(m)}}{\sqrt{nm}} \approx
\sum_{\substack{\substack{d \equiv 0(Q)\\ d\sim PQ\\ d \nequiv 0(P)}}}\sum_{\substack{0<|r| \ll P\\n \sim P^2Q\\ m\sim PQ} }\sum_{c\sim C}\frac{1}{cC}\sideset{}{^*}\sum_{a(c)}e\(\frac{a(Pm-n-rd)}{c}\)\frac{\lambda_{f^{*}}(n)\overline{\lambda_f(m)}}{\sqrt{nm}}.
\end{align}

Since we are using Jutila's circle method, one can assume that $C$ is sufficiently large and $(c,P) =1$ to simplify our proof.\\

\textbf{Step 6. Applying Vornoi formula twice to regenerate Kloostermann sums. }
Then, we apply Voronoi formula (Lemma \ref{lemV}) twice for both $n,m$ to get
\begin{align}
\sum_{\substack{\substack{d \equiv 0(Q)\\ d\sim PQ\\ d \nequiv 0(P)}}}\sum_{\substack{0<|r| \ll P\\n \sim C^2/PQ\\ m\sim C^2/Q} }\sum_{c\sim C}\frac{1}{cC}S\(\overline{P^2}(m - Pn), rd; c\)\frac{\lambda_{f^{*}}(n)\overline{\lambda_f(m)}}{\sqrt{nm}}.
\end{align}

Next, set $v = m-Pn$, so $v \ll C^2/Q$ and our sum becomes
\begin{align}
\sum_{\substack{\substack{d \equiv 0(Q)\\ d\sim PQ\\ d \nequiv 0(P)}}}\sum_{\substack{0<|r| \ll P} }\sum_{k\sim C^2/Q}\sum_{c\sim C}\frac{1}{cC}S\(\overline{P^2}v, rd; c\)\(\sum_{m-Pn = v}\frac{\lambda_{f^{*}}(n)\overline{\lambda_f(m)}}{\sqrt{nm}}\).
\end{align}

\textbf{Step 7. Applying the large sieve type inequality to the sum of Kloostermann sums. }
Now, in order to apply the large sieve type inequality (Proposition \ref{pro1}), we assume that $h = rd/Q $, $w = Q$, $r = P^2$, $v = v$ and $s=1$. The sum becomes
\begin{align}
\sum_{h \sim P^2}b(h)\sum_{v\sim C^2/Q}\sum_{c\sim C}\frac{1}{cC}S\(\overline{P^2}v, hQ; c\)\(\sum_{m-Pn = v}\frac{\lambda_{f^{*}}(n)\overline{\lambda_f(m)}}{\sqrt{nm}}\)
\end{align}
where $b(h) = \sum_{\substack{rd\' = h\\ r,d\'\sim P\\ d\' \nequiv 0(P)}}1 \leq \tau(h)$.\\

Notice that 
\begin{align}\label{role_of_supnorm}
\sum_v\(\sum_{m-Pn = v}\frac{\lambda_{f^{*}}(n)\overline{\lambda_f(m)}}{\sqrt{nm}}\)^2 \ll P^{\frac{2}{3}+\epsilon}
\end{align}
by the argument in Section \ref{step6}. And 
$$\sum_{h\sim P^2}|b(h)|^2 \leq \sum_{h\sim P^2} \tau(h) \ll P^{2+\epsilon}.$$
Thus, recalling that $C$ is sufficiently large, $R_f$ is bounded by 
$$\frac{\sqrt{P^2}}{C}\(1+\sqrt{\frac{C^2Q^{-1}}{P^2}}\)\(1+\sqrt{\frac{P^2}{P^2}}\)Q^{\theta}\sqrt{P^{\frac{2}{3}+\epsilon}}\sqrt{P^{2+\epsilon}}\(PQC\)^{\epsilon} \ll \frac{P^{\frac{4}{3}}}{Q^{\frac{1}{2}-\theta}}\(PQC\)^{3\epsilon}.$$

\begin{rem}
In the last step above, through the large sieve inequality in \cite{DI}, one can only get convexity bound. So our generalization of large sieve inequality is crucial. \eqref{role_of_supnorm}, which is a direct consequence of nontrivial sup-norm bound Proposition \ref{sup}, plays an essential role here. Since we need our final bound to be less than $P^{\frac{3}{2}-\theta}Q^{\theta-\frac{1}{2}}$, a sup-norm bound better than $P^{\frac{1}{2}-\theta}$ is required ( the trivial one is $P^{\frac{1}{2}}$ under our normalization ).\\
\end{rem}

\section{The Proof of Proposition \ref{thm1}}
\label{se2}
Let $f$ be a newform with level $N$. Let the fixed number $k$ (resp. $t_f$) be the weight (resp. archimedean parameter) of $f$ when $f$ is holomorphic (resp. Maa{\ss}).  Also assume that $N$ is square-free, $M$ and $N$ are coprime. Let $h$ be a smooth function, compactly supported on $[\frac{1}{2},\frac{5}{2}]$ such that $h^{(j)}\ll Z_h^{j}$ for a positive constant $Z_h$. Let $X\geqslant 1$. 

\subsection{\textbf{Step 1. Reducing to sum of Kloosterman sums via trace formula. }}
By Petersson formula (Lemma \ref{lemPT}), we have
\begin{align*}
 & \sum_{g \in B_{\kappa}(M)} \omega_g^{-1}\left| \sum_{n \geqslant 1}\frac{\lambda_f(n)\lambda_g(n)}{\sqrt{n}}h\(\frac{n}{X}\)\right|^2\\
= & \sum_{n}\frac{|\lambda_f(n)|^2}{n}h^2\(\frac{n}{X}\) + \sum_{d\equiv 0(M)}\sum_{1 \leqslant n,m }\frac{1}{d}S(m,n,d)\frac{\lambda_f(m)\overline{\lambda_f(n)}}{\sqrt{nm}}J_{\kappa-1}\(\frac{4\pi\sqrt{mn}}{d}\)h\(\frac{n}{X}\)h\(\frac{m}{X}\).
\end{align*}

Due to Rankin's bound, the first term is bounded by $O_{\epsilon}(N^{\epsilon})$. Now we consider the second term.

First, we use the partition of unity $1=\sum_{D}h_0\(\frac{d}{D}\)$ 
, where $D$ runs over values $2^{v}$ with $v = -1, 0, 1, 2, \dots$. Furthermore, we can assume that $h_0(x)$ is a smooth function, compactly supported on $[\frac{1}{2},\frac{5}{2}]$ with bounded derivatives. Finally, let

$$R_{f}(X,D) : = \sum_{\substack{d\equiv0(M)}}\sum_{n,m }\frac{1}{d}S(m,n;d)\frac{\lambda_f(m)\overline{\lambda_f(n)}}{\sqrt{nm}} h_0\(\frac{d}{D}\)h\(\frac{m}{X}\)h\(\frac{n}{X}\)J_{\kappa-1}\(\frac{4\pi\sqrt{mn}}{d}\),$$
such that 
\begin{align}\label{RR_0}
\sum_{g \in B_\kappa(M)} \omega_g^{-1}\left| \sum_{n \geqslant 1}\frac{\lambda_f(n)\lambda_g(n)}{\sqrt{n}}h\(\frac{n}{X}\)\right|^2 = O((XN)^{\epsilon}) + \sum_{D} R_f(X,D)
\end{align}
where $D$ runs over values $2^{v}$ with $v = -1, 0, 1, 2, \dots$.

\subsection{\textbf{Step 2. Removing large and small values of $D$. }}
In this section, we prove the following,
\begin{lem}
\begin{align}
\label{2.3}
&\sum_{g \in B_\kappa(M)} \omega_g^{-1}\left| \sum_{n\geqslant1}\frac{\lambda_f(n)\lambda_g(n)}{\sqrt{n}}h\(\frac{n}{X}\)\right|^2 \ll_{\epsilon,t_f,\kappa} \(1 +M^{-\beta}\frac{X}{M}\)(XMN)^{\epsilon} + \max_{X(M)^{-\beta} < D < X(M)^{2\beta}}\left|R_f(X,D)\right| (XMN)^{\epsilon}. 
\end{align}
\end{lem}
By \eqref{RR_0}, it suffices to estimate $R_f(X,D)$ when $D$ is large and small. 
\subsubsection{Eliminating $R_f(X,D)$ when $D$ is large}
Assume that $D > X(M)^{2\beta}$ for some positive $\beta$.

Let $x_n = \frac{\lambda_f(n)}{\sqrt{n}}h\(\frac{n}{X}\)$ and $y_m = \frac{\lambda_f(m)}{\sqrt{m}}h\(\frac{m}{X}\)$. We apply Lemma \ref{lemLS} to $R_f(X,D)$.  Therefore
\begin{align}
R_f(X,D) &=  \sum_{\substack{d\equiv0(M)}}\sum_{m,n }\frac{h_0\(\frac{d}{D}\)}{d}S(m,n;d)\overline{x_n}y_m J_{\kappa-1}\(\frac{4\pi\sqrt{mn}}{d}\)\label{2.1}\ll_{\epsilon, \kappa}D^{\epsilon}\(\frac{X}{D}\)^{1/2}\(1+\frac{X}{M}\)\|x\|_2\|y\|_2\notag\\
& \ll_{\epsilon, \kappa} M^{-\beta}\(1+\frac{X}{M}\)(XM)^{\epsilon} \sum_n\frac{|\lambda_f(n)|^2}{n}h^2\(\frac{n}{X}\)\notag \ll_{\epsilon, \kappa}\(1+M^{-\beta}\frac{X}{M}\)(XMN)^{\epsilon}.\notag 
\end{align}

\subsubsection{Eliminating $R_f(X,D)$ when $D$ is small}
Now, assume that  $D<X(M)^{-\beta}$ for the same $\beta$ as in the previous case.\\

For fixed $m,n$, consider the test function $$W_{m,n}(x) := J_{\kappa-1}(x)h_0\(\frac{4\pi\sqrt{mn}}{Dx}\),$$ and rewrite $R_f(X,D)$ as
\begin{align*}
\sum_{d\equiv 0(M)}\sum_{m,n}\frac{1}{d}S(m,n;d)\frac{\lambda_f(m)\overline{\lambda_f(n)}}{\sqrt{mn}}h\(\frac{m}{X}\)h\(\frac{n}{X}\)W_{m,n}\(\frac{4\pi\sqrt{mn}}{d}\).
\end{align*}

Notice that $W_{m,n}$ is supported on the interval $[\frac{\pi X}{D}, \frac{10\pi X}{D}]$. Applying Kuznietsov Trace Formula (Lemma \ref{lemKT} with $\fa$ and $\fb$ equaling cusp at $\infty$, $\Gamma=\Gamma(M)$) to $R_f(X,D)$, we obtain
\begin{align}
&R_f(X,D) = \sum_{m,n}\frac{\lambda_f(m)\overline{\lambda_f(n)}}{\sqrt{mn}}h\(\frac{m}{X}\)h\(\frac{n}{X}\)\times\\
&\Bigg(\sum_{j=1}^{\infty}\frac{\hat{W}(t_j)}{\cosh \pi t_j}\rho_j(m)\overline{\rho_j(n)}+ \sum_{\mathfrak{c}}\frac{1}{\pi}\int_{-\infty}^{\infty}\(\frac{n}{m}\)^{-it}\hat{W}(t)\varphi_{\mathfrak{c}}(m,\frac{1}{2} + it)\overline{\varphi_{\mathfrak{c}}(n,\frac{1}{2} + it)}dt\\
&+ \frac{1}{2\pi}\sum_{0<l\equiv 0 (2)}\frac{i^l(l-1)!}{(4\pi)^{l-1}}\tilde{W}(l-1)\sum_{1 \leq j \leq \dim S_l(\Gamma)}\psi_{jl}(m)\overline{\psi_{jl}(n)}\Bigg).
\end{align}

By Lemma \ref{lem12} with $H(x) = h_0\(\frac{4\pi\sqrt{mn}}{Dx}\)$, we know that when $D<X(M)^{-\beta}$ the weight functions satisfy
\begin{subnumcases}{\hat{W}(t), \tilde{W}(t)  \ll}
\(\frac{X}{D}\)^{-1+\epsilon} & $\text{if } t > 0 \text{ or } t=ir$, \label{2.2_1}\\ 
\(\frac{X}{Dt^{2}}\)^A\frac{X}{D}&$ \text{if } t>4t_f+2$.\label{2.2}
\end{subnumcases} 

Consider first the contribution from the sum over the Maass forms. By the Cauchy-Schwarz inequality, we have that
\begin{align}
\label{eqrho}
&\sum_{m,n}\frac{\lambda_f(m)\overline{\lambda_f(n)}}{\sqrt{mn}}h\(\frac{m}{X}\)h\(\frac{n}{X}\)\sum_{j=1}^{\infty}\frac{|\hat{W}(t_j)|}{\cosh \pi t_j}\rho_j(m)\overline{\rho_j(n)}\\
 &\ll \sum_T \sum_{T\leq|t_j|< 2T}\left|\frac{\hat{W}(t_j)}{\cosh \pi t_j}\right|\left|\sum_n \frac{\lambda_f(n)}{\sqrt{n}}h\(\frac{n}{X}\)\rho_j(n)\right|^2\\
 &\ll\(\sum_{T\leq \(\frac{X}{D}\)^{1/2 + \epsilon}} + \sum_{T>\(\frac{X}{D}\)^{1/2 + \epsilon}}\) \sum_{T\leq|t_j|< 2T}\max_{T\leq|t_j|< 2T}\{|\hat{W}(t_j)|\}\(T^2 + \frac{X}{M}\)(XM)^{\epsilon}\sum_n\frac{|\lambda_f(n)|^2}{n}h^2\(\frac{n}{X}\)
\end{align}
where $T$ goes over powers of $2$. For the last step above, we used Lemma \ref{lem13}.

Applying the bounds for $\hat{W}(t)$, we therefore obtain an upper bound for \eqref{eqrho}
\begin{align*}
&\ll_{\epsilon,\kappa} \(\frac{X}{D}\)^{-1+\epsilon}\(\(\frac{X}{D}\)^{1 + 2\epsilon} + \frac{X}{M}\)(XMN)^{\epsilon} + \(\frac{X}{D}\)^{-100}\(1 + \frac{X}{M}\)(XMN)^{\epsilon}\\
& \ll_{\epsilon,\kappa} \(1 + M^{-\beta}\frac{X}{M}\)(XMN)^{\epsilon}.
\end{align*}

Similarly, we have the same bound for $\varphi_{\mathfrak{c}}$ and $\psi_{jk}$. Therefore
\begin{align}
R_f(X,D)\ll_{\epsilon,\kappa} \(1 + M^{-\beta}\frac{X}{M}\)(XMN)^{\epsilon}.\label{21}
\end{align}

\subsection{\textbf{Step 3. Applying the Voronoi formula to convert Kloosterman Sums to Ramanujan sums. }}
Let $D$ be such that 
\begin{equation}
X(M)^{-\beta} < D < X(M)^{2\beta}\label{D}.
\end{equation}
As was done in \cite{HM}, we will apply the Voronoi formula (Lemma \ref{lemV}) to the $m$-sum. Assume that $(d,N)=R$ and $LR=N$. Set $N_2 = \frac{N}{(N,d)}=L$. Since $N$ is square-free, $(d,L)=1$. We have
 \begin{align}
R_{f}(X,D) =& 2\sum_{LR=N}\sum_{\pm}\eta_L^{\pm} \sum_{\substack{d\equiv 0(RM)\\ (d,L)=1}}\sum_{n, m}\frac{1}{d}\sideset{}{^*}\sum_{a(d)}\frac{\lambda_{f^*}(m)}{\sqrt{m}}e\(\mp\frac{\overline{aL}m}{d}\)e\(\frac{\overline{a}n}{d}\)\frac{\overline{\lambda_f(n)}}{\sqrt{n}}\label{2.5}\\
&\times h_0\(\frac{d}{D}\)h\(\frac{n}{X}\)\int_{0}^{\infty}h\(\frac{Ld^2t^2}{mX}\)\mathcal{J}^{\pm}_f(t)J_{\kappa-1}\(\frac{4\pi\sqrt{nL}t}{\sqrt{m}}\)dt \notag\\
= &  \sum_{LR=N}\sum_{\pm} \sum_{\substack{d\equiv 0(RM)\\ (d,L)=1}}\sum_{ n , m}\frac{1}{d}S(0,m\mp nL;d)\frac{\overline{\lambda_f(n)}\lambda_{f^*}(m)}{\sqrt{nm}}I^{\pm}_{L,X,D}(m,n,d)\notag\\
=&  \sum_{\pm}\sum_{LR=N} \sum_{\substack{d\equiv 0(RM)\\ (d,L)=1}}\sum_{bc=d}\frac{c\mu(b)}{d}\sum_{\substack{n, m\\m\mp nL =0(c)\\m\mp nL \neq 0}}\frac{\overline{\lambda_f(n)}\lambda_{f^*}(m)}{\sqrt{nm}}I^{\pm}_{L,X,D}(m,n,d) + R_{0}\notag 
\end{align} 
where
\begin{align}
I^{\pm}_{L,X,D}(m,n,d) = 2\eta^{\pm}_L h\(\frac{n}{X}\)h_0\(\frac{d}{D}\)\int_{0}^{\infty}h\(\frac{Ld^2t^2}{mX}\)\mathcal{J}^{\pm}_f(t)J_{\kappa-1}\(\frac{4\pi\sqrt{nL}t}{\sqrt{m}}\)dt,\label{I}
\end{align}
and the zero shift
\begin{align}
R_0 = \sum_{\pm}\sum_{LR=N} \sum_{\substack{d\equiv 0(RM)\\ (d,L)=1}}\frac{\varphi(d)}{d}\sum_{\substack{m\mp nL = 0}}\frac{\overline{\lambda_f(n)}\lambda_{f^*}(m)}{\sqrt{nm}}I^{\pm}_{L,X,D}(m,n,d).
\end{align}
Here we used the identity for Ramanujan sum
\begin{equation}
\label{Rama}
S(0,k;d) = \sum_{c|(d,k)}c\mu\(\frac{d}{c}\).
\end{equation}

\begin{rem}
As stated in Lemma \ref{lemV}, $\eta^{\pm}_L$ only depends on $L$ and $f$ and has norm $1$ in this case. 
\end{rem}


From the estimation of $I^{\pm}_{L,X,D}(m,n,d)$ in Lemma \ref{lem9}, we have that the contribution from those $m$ satisfying either
\begin{align*}
\left|\sqrt{\frac{mX}{LD^2}}-\frac{\sqrt{nX}}{D}\right| \gg\(1+Z_h\)M^{\epsilon} 
\end{align*}
for $f$ holomorphic, or
\begin{equation}
\sqrt{\frac{mX}{LD^2}} \gg (1+Z_h)M^{\epsilon} \text{ and } \left|\sqrt{\frac{mX}{LD^2}}-\frac{\sqrt{nX}}{D}\right| \gg\(1+Z_h\)M^{\epsilon}
\end{equation}
for $f$ Maa{\ss} are negligible. Thus we can truncate $m$ such that $m \ll M^{\epsilon}\lambda LD$ where $\lambda$ is defined to be
\begin{align}
\lambda = \max \left\{\frac{D}{X}, \frac{X}{D}\right\}\(1+Z_h\)^2.\label{lambda}
\end{align}
Then via \eqref{D} we have $\lambda \leq M^{2\beta}\(1+Z_h\)^2$. \\

Hence one can break apart the sum over $m$ dyadically such that
\begin{align}
&R_{f}(X,D) =\\
&\sum_{\substack{S=2^{i}\\S\ll M^{\epsilon}\lambda LD\\ LR=N}}\sum_{\substack{d\equiv0(RM)\\ (d,L)=1}}\sum_{\substack{bc=d\\ \pm}}\frac{c\mu(b)}{d}\sum_{\substack{n, m\\m\mp nL \equiv 0(c)}}\frac{\overline{\lambda_f(n)}\lambda_{f^*}(m)}{\sqrt{nm}}I^{\pm}_{L,X,D}(m,n,d)h_0\(\frac{m}{S}\) +R_0+ O(M^{-100}).\label{2.6}
\end{align}

\subsection{\textbf{Step 4. Treating the Zero Shift. }}\label{step4}
In \eqref{2.6}, the inner sum is over all the pairs $(m,n)$ such that $m\mp nL\equiv 0 (c)$. In this section, we will treat the $m-nL=0$ case ( we always have that $m+nL > 0$ ).
\begin{align}\label{zero shift}
R_0 &=\sum_{\substack{ LR=N}}\sum_{\substack{d\equiv0(RM)\\ (d,L)=1}}\sum_{\substack{n}}\frac{\overline{\lambda_f(n)}\lambda_{f^*}(n)\lambda_{f^*}(L)}{\sqrt{L}n}I^{+}_L(nL,n,d)\\
&\ll \frac{X}{NM}\sum_{n\ll X}\frac{\left|\lambda_f(n)\right|^2+\left|\lambda_{f^*}(n)\right|^2}{n}(XMN)^{\epsilon}\ll \frac{X}{NM}(XMN)^{\epsilon},
\end{align}
where the last two steps follow from Lemma \ref{lem9} (when $f$ is non-exceptional), \ref{Rank} and the bound $\left|\lambda_f(L)\right| = L^{-1/2}$ \eqref{saving_norm}.

\subsection{The Sum of Shifted Sums.}
\label{section_soss}
Let $m\mp nL = rc$ in \eqref{2.6}. Since $bc = d$ and $M | d$, let $c_0 : = (c,M)$ such that $c=c\' c_0$. Let $h = rc\'$ and $d = d\' RM$. Then, after rewriting \eqref{2.6} 
\begin{align}
\label{R_f bound}
&R_f(X,D) = \\
&\sum_{\substack{S\\ LR=N}}\sum_{\substack{c_0|M}}\frac{c_0}{RM}\sum_{\substack{m,n,h \neq 0\\ m\mp nL = hc_0}}\sum_{\substack{(d\' , L) =1\\c\' | (d\' R,h)}}\frac{c\' \mu(\frac{MRd\'}{c_0c\'})}{d\'}\frac{\overline{\lambda_f(n)}\lambda_{f^*}(m)}{\sqrt{nm}}I^{\pm}_{L,X,D}(m,n,d\'RM)h_0\(\frac{m}{S}\) + O\(\frac{X^{1+\epsilon}}{(MN)^{1-\epsilon}}\).
\end{align}

We can define $b(h,d\') : = \sum_{\substack{(d\' , L) =1\\c\' | (d\' R,h)}}c\' \mu(\frac{MRd\'}{c_0c\'})d^{\prime-1}$. Therefore, it is natural to study the sum of shifted sums. We have the following proposition.
\begin{pro}
\label{sum_of_shifted_sums}
Let $f, g$ be newforms with the same level $N$. Let $I(x,y,d)$ be a smooth function supported on $[1/2S_1, 5/2S_1]\times [1/2S_2, 5/2S_2] \times [1/2D_1, 5/2D_1]$ with $(x,y,d)$-type $(1: Z_1,Z_2, Z_3)$ (see Section \ref{weight_functions}). Set $Z=Z_1+Z_2+Z_3+1$. Let $b(h,d)$ be a series of complex numbers. Let $c_0$ be an integer coprime with $N$. Then 
\begin{align}
&\sum_{d}\sum_{\substack{m,n,h\neq 0\\ l_1m \mp l_2n = hc_0}}b(h,d)\frac{\overline{\lambda_f(n)}\lambda_{g}(m)}{\sqrt{nm}}I(m,n,d) \ll_{\epsilon, t_f, t_g} \max_{H} \left\{c_0^{\theta - \frac{1}{2}}N^{\frac{5}{6}}\sqrt{l_1l_2}\frac{\|B(h)\|_H}{\sqrt{H}}\(S_{1,2}+\frac{H^{1/2}}{\sqrt{Nl_1l_2}}\)\right\}Z^{11}Q^{\epsilon}
\end{align}
up to a factor of size $(ZS_1S_2l_1l_2Nc_0)^{\epsilon}$, where $S_{1,2} = \sqrt{\(S_1l_1+S_2l_2\)\(\frac{1}{S_1l_1}+\frac{1}{S_2l_2}\)}$, $H$ is any number in the range of $h$ ($Hc_0 \ll S_1l_1+S_2l_2$ holds automatically) and 
$$B(h) := \sum_{d}|b(h,d)|,\, \ \|B(h)\|^2_H : = \sum_{h\sim H} |B(h)|^2.$$
\end{pro}

We will prove this proposition in the following few steps.

\subsection{\textbf{Step 5. Applying the Circle Method. }}
By the support of $I(x,y,d)$, we have $|h|<3(l_1S_1+l_2S_2)/c_0$. Now, we shall apply Jutila's circle method to detect the relation $l_1m \mp l_2n = hc_0$. Let $c_0 \ll M$ for some $M$ (We will choose $M$ to be the level eventually in the case that $c_0$ is small).

As the notations in Section \ref{se1.5}, we choose $\delta = Q^{-1}$,
$Q = (|l_1l_2|S_1S_2D_1MN)^{100}$, and\\
$\mathfrak{Q} = \left\{q : Q < q < 2Q, (q,l_1l_2N) = 1\right\}$. So that $\Lambda \gg Q^{2 - \epsilon}$. Thus, by Jutila's circle method, the inner sum in \eqref{2.6} gives

\begin{align}
&\sum_{n,m}\sum_{0<|h|<3(l_1S_1+l_2S_2)/c_0}b(h,d)\frac{\overline{\lambda_f(n)}\lambda_{g}(m)}{\sqrt{nm}}I(m,n,d)\int_0^1e((l_1m\mp l_2n-hc_0)x)dx\\
= &\frac{1}{\Lambda}\sum_{q \in \mathfrak{P}}\sideset{}{^*}\sum_{a(q)}\sum_{n,m}\sum_{0<|h|<3(l_1S_1+l_2S_2)/c_0} b(h,d)e\(\frac{(l_1m\mp l_2n-hc_0)a}{q}\)\frac{\overline{\lambda_f(n)}\lambda_{g}(m)}{\sqrt{nm}}I(m,n,d)\\
& \times \frac{1}{2\delta}\int_{-\delta}^{\delta}e(\Delta x)dx + \mathcal{E}(l_1,l_2,d,c_0)\label{Main1}
\end{align}
where 
$\Delta = \Delta(l_1m,l_2n,hc_0) = l_1m\mp l_2n-hc_0$, and
\begin{equation}
\begin{split}
\mathcal{E}(l_1,l_2,d,c_0) =& \sum_{n,m,h}b(h,d)\frac{\overline{\lambda_f(n)}\lambda_{g}(m)}{\sqrt{nm}}I(m,n,d)\int_0^1(1-\tilde{I}_{\mathfrak{P},\delta}(x))e((l_1m\mp l_2n-hc_0)x)dx\\
\leq & \sum_{n,m,h}\left|b(h,d)\frac{\overline{\lambda_f(n)}\lambda_{g}(m)}{\sqrt{nm}}I(m,n,d)\right|\int_0^1|1-\tilde{I}_{\mathfrak{P},\delta}(x)|^2dx\\
\leq & \sum_{n,m,h}\left|b(h,d)\frac{\overline{\lambda_f(n)}\lambda_{g}(m)}{\sqrt{nm}}I(m,n,d)\right|\frac{Q^{2+\epsilon}}{\delta \Lambda^2}\leq \|b(h,d)\|_2(l_1l_2S_1S_2D_1MN)^{-50}.
\end{split}
\end{equation}
The last inequality above follows from Lemma \ref{lem9} and \ref{Rank}. Here we used a similar argument to the one in \cite{B1}.
 
Let 
\begin{equation}\label{weight w}
w_{\delta}\(\Delta\) : = \frac{1}{2\delta}\int_{-\delta}^{\delta}e(\Delta x)dx.
\end{equation}
Set
\begin{align}\label{RL}&R_{f,g}: =  \\
& \frac{1}{\Lambda}\sum_{q \in \mathfrak{P}}\sideset{}{^*}\sum_{a(q)}\sum_{n,m,d}\sum_{0<|h|<3(l_1S_1+l_2S_2)/c_0} b(h,d)e\(\frac{(l_1m\mp l_2n-hc_0)a}{q}\)\frac{\overline{\lambda_f(n)}\lambda_{g}(m)}{\sqrt{nm}}I(m,n,d)w_{\delta}\(\Delta\).
\end{align}

\subsection{\textbf{Step 6. Applying Vornoi formula twice to regenerate Kloostermann sums. }}
We only treat the $-$ sign case, the $+$ sign case can be treated similarly. Recall that $(q,l_1l_2N) =1$ for any $q \in \mathfrak{Q}$. So that $N_2 = N/(N,q) = N$. Also, we have that $\Lambda \gg Q^{2-\epsilon}$. 

Applying Voronoi formula (Lemma \ref{lemV}) twice to both $m,n$-sum in \eqref{RL}, we get 
\begin{align}
&R_{f,g}
=\frac{1}{\Lambda}\sum_{\pm_1,\pm_2}\sum_{q \in \mathfrak{Q}}\sum_{n,m,d}\sum_{h}S((\mp_1 l_2m \pm_2 l_1n)\overline{l_1l_2N}, hc_0; q) \frac{\overline{\lambda_{f^*}(n)}\lambda_{g^*}(m)}{\sqrt{nm}}H^{\pm_1, \pm_2}(m,n,h,d,q), \label{RfL}
\end{align}
where $H^{\pm_1, \pm_2}(m,n,h,d,q)$ is given by
\begin{align}\label{H_L}
&H^{\pm_1, \pm_2}(m,n,h,d,q) :=\\
&\iint_{0}^{\infty}4 \eta_g^{\pm_1}\eta_f^{\pm_2} I\(\frac{\xi^2q^2N}{m},\frac{\mu^2q^2N}{n},d\)w_{\delta}\(\Delta\(\frac{\xi^2q^2N}{m},\frac{\mu^2q^2N}{n},hc_0\)\)\mathcal{J}_g^{\pm_1}(\xi)\mathcal{J}_f^{\pm_2}(\mu)d\xi d\mu.
\end{align}
and $\eta_f^{\pm_2}, \eta_g^{\pm_1}$ depend on $f,g,N$ only.

Recall that $\delta = Q^{-1}$,
$Q = (|l_1l_2|S_1S_2D_1MN)^{100}$.
By Lemma \ref{lem10}, we can truncate the sum over $m,n$ such that
\begin{align*}
m  \ll \frac{q^2N(1+Z_1(S_1))}{S_1}(Q)^{\epsilon}, n \ll \frac{q^2N(1+Z_2(S_2))}{S_2}(Q)^{\epsilon}.
\end{align*}

Breaking apart the $m,n$-sum dyadically, we can assume that the sizes of $m,n$ are $A,B$ respectively with $A  \ll \frac{q^2N(1+Z_1)}{S_1}(Q)^{\epsilon}$ and $B \ll \frac{q^2N(1+Z_2)}{S_2}(Q)^{\epsilon}.$

Let
\begin{align}
&\widetilde{R}(A,B) : =\frac{1}{\Lambda}\sum_{\pm_1,\pm_2}\sum_{q \in \mathfrak{Q}}\sum_{n,m,d}\sum_{h}b(h,d)S((\mp_1 l_2m \pm_2 l_1n)\overline{l_1l_2N}, hc_0; q) \frac{\overline{\lambda_{f^*}(n)}\lambda_{f}(m)}{\sqrt{nm}}\widetilde{H}_{A,B}^{\pm_1, \pm_2}(m,n,h,d,q), \label{RfL}
\end{align}
where 
\begin{equation}
\widetilde{H}_{A,B}^{\pm_1, \pm_2}(m,n,h,d,q)=H^{\pm_1, \pm_2}(m,n,h,d,q)h_0\(\frac{m}{A}\)h_0\(\frac{n}{B}\).
\end{equation}
Then, we have
\begin{equation}
\label{R'}
R_{f,g} \ll  \max_{A,B}\{\widetilde{R}(A,B)\}((1+Z_1(S_1))(1+Z_2(S_2))Q)^{\epsilon}.
\end{equation}

The $\sum_{\pm_1, \pm_2}$ contains four terms. We only consider both $+$ case, and the proofs of other cases will be similar. Let $l_2m - l_1n =v$. Set $h^{\pm}(v)$ be functions such that $h^{\pm}(v)=1$ when $\pm v\geqslant 2/3$ and $h^{\pm}(v)=0$ when $\pm v \leqslant 1/3$. By Abel's summation formula, we have
\begin{align}
\widetilde{R}^{+,+}(A,B)
= &\frac{1}{\Lambda}\sum_{d}\sum_{q\in\mathfrak{Q}}\sum_{h}b(h,d)S(0, hc_0; q) \sum_{l_2m=l_1n }\frac{\overline{\lambda_{f^*}(n)}\lambda_{g^*}(m)}{\sqrt{nm}}\widetilde{H}^{+,+}_{A,B}(m,n,h,d,q)\\
& -\frac{1}{\Lambda}\sum_{d}\sum_{q\in\mathfrak{Q}}\sum_{h}b(h,d) \int_{\frac{1}{2}}^{\infty}\sum_{v>0}S(v\overline{l_1l_2N}, hc_0; q) \sum_{m \leqslant x, l_2m-l_1n = v}\frac{\overline{\lambda_{f^*}(n)}\lambda_{g^*}(m)}{\sqrt{nm}}u^+(v,h,q; d,x) dx\\
& - \frac{1}{\Lambda}\sum_{d}\sum_{q\in\mathfrak{Q}}\sum_{h}b(h,d) \int_{\frac{1}{2}}^{\infty}\sum_{v<0}S(v\overline{l_1l_2N}, hc_0; q) \sum_{m \leqslant x, l_2m - l_1n =v}\frac{\overline{\lambda_{f^*}(n)}\lambda_{g^*}(m)}{\sqrt{nm}}u^-(v,h,q;d,x) dx\\
= &\, \ \mathfrak{I}_0 - \mathfrak{I}_+ - \mathfrak{I}_-,\label{2.8}
\end{align}
where $u^-(v,h,q; d,x) : = \frac{d}{d x}\widetilde{H}^{+,+}_{A,B}(x,\frac{l_1x-v}{l_2},h,d,q)h^{-}(v)$ and $u^+(v,h,q;d,x) : = \frac{d}{d x}\widetilde{H}^{+,+}_{A,B}(\frac{v+l_1x}{l_2},x,rc,d,q)h^{+}(v)$.

Let $$A(v;x) : = \sum_{\substack{m \leqslant x\\ l_1m - l_2n =v}}\frac{\overline{\lambda_{f^*}(n)}\lambda_{f}(m)}{\sqrt{nm}}.$$
\begin{rem}
In Theorem \ref{pro1}, $w$ need to be coprime with the level $rs$. In our case, $c_0$ is coprime with the level $Nl_1l_2$ that we will consider. So, $c_0$ will play the role as $w$, which gives us the major saving. 
\end{rem}


\subsection{\textbf{Step 7. Applying the large sieve type inequality to the sum of Kloostermann sums. } }\label{step6}
In this section, we will bound $\mathfrak{I}_0$ and $\mathfrak{I}_{\pm}$.


First, consider $\mathfrak{I}_0$ as defined in \eqref{2.8}. By \eqref{Rama}, Lemma \ref{lem10}, the Cauchy-Schwarz inequality and the Rankin's bound (Lemma \ref{Rank}), we have
\begin{align}
\mathfrak{I}_0 & =\frac{1}{\Lambda}\sum_{d}\sum_{q\in\mathfrak{Q}}\sum_{h}b(h,d)\sum_{u|(hc_0,q)} \mu\(\frac{q}{u}\)u \sum_{l_1m=l_2n }\frac{\overline{\lambda_{f^*}(n)}\lambda_{g^*}(m)}{\sqrt{nm}}\widetilde{H}_{A,B}(m,n,h,d,q) \label{2.10}\\
& \ll_{\epsilon} \frac{1}{Q^{1-\epsilon}}\sum_{h,d}\tau(hc_0)b(h,d)\sum_{n \ll l_1A+l_2B}\frac{|\lambda_f(n)|^2}{n}\notag \ll_{\epsilon} \|B(h)\|_2Q^{-1/2} 
\end{align}

Secondly, consider $\mathfrak{I}_+$. We use the notation in Lemma \ref{lem11}, which denotes the $(v, h, q, d)$-type of $u^+(v,h,q; d ,x)$ as
$$\(B_{u^{+}} : Z_1+1, 1, Z_1+Z_2+1, Z_3+1\),$$
and $20Al_2 > Bl_1$ when $u^{+}$ is nonzero. 


Break apart the $v,h$-sum dyadically such that $v \sim V$, $h \sim H$. Let $\Xi=\frac{\sqrt{VHc_0}}{\sqrt{Nl_1l_2}Q}$. Then, apply Theorem \ref{pro1} to obtain that
\begin{align}
\label{I1}
 \mathfrak{I}_+ &= \frac{1}{\Lambda}\sum_{d}\sum_{q\in\mathfrak{Q}}\sum_{h}b(h,d) \int_{\frac{1}{2}}^{\infty}\sum_{v>0}S(v\overline{l_1l_2N}, hc_0; q) \sum_{\substack{m \leqslant x\\ l_2m-l_1n = v}}\frac{\overline{\lambda_{f^*}(n)}\lambda_{g^*}(m)}{\sqrt{nm}}u^+(v,h,q; d,x) dx\\
&\ll  \max_{V,H}\{\frac{\sqrt{Nl_1l_2}}{\(Z+\Xi\)Q}B_{u^+}B\left(1+\left(\Xi\right)^{-2\theta}\right) \left(Z+\Xi+ \frac{V^{1/2+\varepsilon}}{\sqrt{Nl_1l_2}}\right)\\
&\times\left(Z+\Xi+ \frac{H^{1/2+\varepsilon}}{\sqrt{Nl_1l_2}}\right)c_0^{\theta}\max_{x}\{||A(v,x)||\}\|B(h)\|_H\}Z^8Q^{\epsilon}
\end{align}
where $B(h) = \sum_{d}|b(h,d)|$, and 
$$\|B(h)\|^2_H = \sum_{h \sim H}A^2(h).$$

Since $V$ is the size of $v = ml_2 - nl_1$ and $u^{+}$ is nonzero only if $v>0$ , we have that $V \ll Al_2 $. Furthermore, we have
\begin{align}
&||A(v,x)||^2 =  \sum_{\substack{m_1, m_2 \leqslant x\\ l_2m_1-l_1n_1 = l_2m_2-l_1n_2}}\frac{\overline{\lambda_{f^*}(n_1)}\lambda_{f}(m_1)\lambda_{f^*}(n_2)\overline{\lambda_{f}(m_2)}}{\sqrt{n_1m_1n_2m_2}}\label{Ak}\\
&=  \int_{0}^{1}\sum_{m_1,m_2 \leqslant x }\frac{\lambda_{f}(m_1)e(m_1l_2\alpha)\overline{\lambda_{f}(m_2)}e(-m_2l_2\alpha)}{\sqrt{m_1m_2}}\\
&\times \sum_{n_1,n_2 \leqslant l_2x/l_1}\frac{\overline{\lambda_{f^*}(n_1)}e(-n_1l_1\alpha)\lambda_{f^*}(n_2)e(n_2l_1\alpha)}{\sqrt{n_1n_2}} d \alpha,
\end{align}
which, by Lemma \ref{Wilton}, is
\begin{align}
&\ll \|S(f,x,l_2\alpha)\|_{\infty}^2\int_{0}^{1}\left|\sum_{n \leqslant l_2x/l_1}\frac{\lambda_{f^*}(n)e(-nl_1\alpha)}{\sqrt{n}}\right|^2 d \alpha\\
&\ll_{\epsilon} N^{2/3}(Nx)^{\epsilon}\sum_{n \leqslant l_2x/l_1}\frac{|\lambda_{f^*}(n)|^2}{n} \ll_{\epsilon} N^{2/3}(xNl_1l_2)^{\epsilon}.
\end{align}


By Lemma \ref{lem11}, \eqref{I1}, 
 and \eqref{Ak}, recalling that $A \ll (1+Z_1)Q^{2+\epsilon}N/S_1$and $B \ll (1+Z_2)Q^{2+\epsilon}N/S_2$, 
$\mathfrak{I}_+$ is bounded by
 \begin{align}
 &\max_{\substack{V,H}}\left\{\frac{c_0^{\theta}N^{5/6}Z^{9}\sqrt{l_1l_2}}{\(Z+\Xi\)Q}\left(1+\left(\Xi\right)^{-2\theta}\right) \left(Z+\Xi+ \frac{V^{1/2+\varepsilon}}{\sqrt{Nl_1l_2}}\right)\left(Z+\Xi+ \frac{H^{1/2+\varepsilon}}{\sqrt{Nl_1l_2}}\right)\|B(h)\|_H\right\}Q^{\epsilon}.
 \end{align}
When $VHc_0 \leqslant Nl_1l_2Q^2$, recalling that $H \ll (S_1l_1+S_2l_2)/c_0$ and $Q$ is sufficiently large , we have
\begin{align}
\mathfrak{I}_+  
\ll_{\epsilon} &\max_{\substack{V,H}}\left\{\frac{c_0^{\theta}N^{5/6}Z^{8}\sqrt{l_1l_2}}{Q}\(\frac{Nl_1l_2Q^2}{VHc_0}\)^{\theta}\left(Z+ \frac{V^{1/2+\varepsilon}}{\sqrt{Nl_1l_2}}\right)\left(Z+ \frac{H^{1/2+\varepsilon}}{\sqrt{Nl_1l_2}}\right)\right\}(Q)^{\varepsilon}\\
\ll_{\epsilon} &\max_{\substack{H\ll (S_1l_1+S_2l_2)/c_0}}\left\{\frac{c_0^{\theta-\frac{1}{2}}N^{5/6}\sqrt{l_1l_2}\|B(h)\|_H}{\sqrt{H}}\left(1+ \frac{H^{1/2+\varepsilon}}{\sqrt{Nl_1l_2}}\right)\right\}Z^{8}Q^{\epsilon}.
\end{align}
When $VHc_0 > NLQ^2$, recalling that $V \ll Al_2 \ll ZQ^{2+\epsilon}Nl_1l_2\((S_1l_1)^{-1}+(S_2l_2)^{-1}\)$, we have
\begin{align}
\mathfrak{I}_+  
\ll_{\epsilon} &\max_{\substack{H}}\left\{\frac{c_0^{\theta-\frac{1}{2}}N^{5/6}\sqrt{l_1l_2}\|B(h)\|_H}{\sqrt{H}}\left(S_{1,2}+ \frac{H^{1/2+\varepsilon}}{\sqrt{Nl_1l_2}}\right)\right\}Z^{11}Q^{\epsilon},
\end{align}
where $S_{1,2} = \sqrt{\(S_1l_1+S_2l_2\)\(\frac{1}{S_1l_1}+\frac{1}{S_2l_2}\)}$.

The estimation of $\mathfrak{I}_-$ is similar. Thus, we combine these bounds with \eqref{R'} and \eqref{2.8} to compete the proof of Proposition \ref{sum_of_shifted_sums}.

\subsection{Conclusion and the Final bound. }
By the results of Section \ref{section_soss}, we can apply Propostion \ref{sum_of_shifted_sums} to $R_f(X,D)$ with 
$I(x,y,d\') = I^{\pm}_{L,X,D}(m,n,d\'RM)h_0\(\frac{m}{S}\)$ and $b(h,d\') : = \sum_{\substack{(d\' , L) =1\\c\' | (d\' R,h)}}c\' \mu(\frac{MRd\'}{c_0c\'})d^{\prime-1}$. For a fixed $c_0$,
\begin{align}
\|B(h)\|_H = \sum_h \left|\sum_{\substack{ d\'\sim D/RM\\(d\',L)=1}}|b_{c_0}(h,d\'R)|\right|^2 \leq \sum_h \left|\sum_{(d\',L)=1}\sum_{c\'|(d\'R,h)}\frac{c\'}{d\'R}\right|^2  
\ll  \sum_{h}\left|\sum_{s|h}\sum_{d\'R\equiv 0 (s)}\frac{s}{d\'R}\right|^2
\ll  H(HDR)^{\epsilon}\label{Bh}.
\end{align}

In our case, we have that $S_1 =S, S_2=X, D_1=D/RM, l_1=1, l_2=L$. Moreover, recall that $S\ll M^{\epsilon}\lambda LD$. By Propostion \ref{sum_of_shifted_sums} and \eqref{R_f bound}, we have
$$R_f(X,D) \ll_{\epsilon,t_f,\kappa} \(\frac{X}{MN}+\frac{\lambda^{12}N^{4/3}}{M^{1/2-\theta}}\(1+\sqrt{\frac{X}{MN}}\)\)(\lambda XMN)^{\epsilon}.$$

\begin{rem}
We used that $f$ is non-exceptional here. Otherwise, we would have a large loss when $S$ is small. 
\end{rem}


Using the result of Kim-Sarnak in \cite{K}, we take $\theta = \frac{7}{64}$ and choose $\beta = \frac{11}{4875}$. Then, from \eqref{2.3} and $\lambda < M^{2\beta}\(1+Z_h\)^2$, we have
\begin{align}
\sum_{g \in B_{\kappa}(M)} \omega_g^{-1}\left| \sum_{n \geqslant 1}\frac{\lambda_f(n)\lambda_g(n)}{\sqrt{n}}h\(\frac{n}{X}\)\right|^2 &\ll_{\epsilon,t_f,\kappa} \(1+\frac{X}{MN}+\frac{X}{M^{1+\beta}}+\frac{\(1+Z_h\)^{24}N^{4/3}}{M^{1/3+\beta}}\(1+\sqrt{\frac{X}{MN}}\)\)(XMN)^{\epsilon}.
\end{align}
We finish the proof of Proposition \ref{thm1}.

\begin{rem}\label{key remark}
The choice of $\beta$ is not optimal here when $X\sim MN$, $N<M$. In this case, one can choose $25\beta = \frac{1}{2}-\theta-\frac{\log_MN}{3}$. As a result, the bound would be  $1+ N^{76/75}M^{-1/64}$. 
\end{rem}

\begin{rem}
In this proof, $M$ is not necessarily square-free. It is also possible to show that when $(M,N)$ is very small, a subconvexity bound still holds. 
\end{rem}

\section{The Estimation of Weight Functions}
\label{weight_functions}

In this section, we will use the lemmas in section \ref{se1.2} to study various weight functions appearing in our analysis in section \ref{se2} such as $I_{L,X,D}(x,y,d)$, $H_{L}(m,n,h,d,q)$ and $u^{\pm}(v,h,q;d,x)$.\\

In order to simplify our notation, we introduce the following definition for the "type" of a function.
\begin{de}
Let $\mathbf{x}=(x_1,x_2,\dots,x_n)$ be a vector of real numbers with each $x_i \neq 0$. Let $F(\mathbf{x})$ be a function. If there are nonnegative functions $Z_F(\mathbf{x}), F_1(\mathbf{x}),F_2(\mathbf{x}),\dots ,F_n(\mathbf{x})$ such that  $$|x_1^{i_1}\dots x_n^{i_n}\partial_{x_1}^{i_1}\dots \partial_{x_n}^{i_n}F(\mathbf{x})|\ll_{i_1\dots i_n} Z_F(\mathbf{x})F_1(\mathbf{x})^{i_1}\dots F_n(\mathbf{x})^{i_n}$$ for every $\mathbf{x},$ where the implied constant depends on $i_1,\dots,i_n$ only, then we call $F(\mathbf{x})$ has $\mathbf{x}=(x_1,x_2,\dots,x_n)$-type $$(Z_F(\mathbf{x}): F_1(\mathbf{x}),F_2(\mathbf{x}),\dots ,F_n(\mathbf{x})).$$

Moreover, when $F(\mathbf{x})$ does not depend on $x_l$ for some $l$, we let $F_l(\mathbf{x})=0$.
\end{de}

For example, let $F(x,y):=e(x)$. Since $x^j\partial_x^{j}F(x,y) = (2\pi i)^{j}x^je(x)$ and $F(x,y)$ is independent of $y$, we have that $F(x,y)$ has $(x,y)$-type $(1: |x|, 0)$.\\

First of all, we establish some basic properties about types. These properties will be used in the study of our weight functions. 

Let $F$ and $G$ be $\mathbb{R}^m-\text{to}-\mathbb{R}$ functions with types $(Z_F:F_1,\dots,F_m)$ and $(Z_G:G_1,\dots,G_m)$ with each $F_i, G_i$ nonnegative. Let $\mathbf{ x} = (x_1,\dots,x_m)$ with each $x_i \neq 0$.
\begin{lem} \label{lem1}
$\partial_{x_k}F(\mathbf{ x})$  has $\mathbf{ x}$-type $(Z_FF_k/x_k:F_1,\dots,F_m)$.
\end{lem}

\begin{proof}
Without loss of generality, let $k=1$.
$$x_{1}^{i_1}x_{2}^{i_2}\dots x_{m}^{i_m}\partial_{x_1}^{i_1} \dots \partial_{x_m}^{i_m}\partial_{x_1}F(\mathbf{ x}) = \frac{1}{x_1}x_{1}^{i_1+1}x_{1}^{i_1}\dots x_{m}^{i_m}\partial_{x_2}^{i_2} \dots\partial_{x_m}^{i_m}\partial_{x_1}^{i_1+1} F(\mathbf{ x}) \ll \frac{Z(\mathbf{ x})F_{1}(\mathbf{x})}{x_1} F_{1}(\mathbf{ x})^{i_1}\dots F_{m}(\mathbf{ x})^{i_m}$$
 \end{proof}

We now use Lemma \ref{lem1} to establish the type of $F(\mathbf{ x})G(\mathbf{ x})$.

\begin{lem} \label{lem2}
$F(\mathbf{ x})G(\mathbf{ x})$ has  $\mathbf{ x}$-type $(Z_FZ_G:F_1+G_1,\dots,F_m+G_m)$.\\
\end{lem}

\begin{proof}
Induction on $(i_1,i_2,\dots, i_m)$, assume that for any $(i_1,i_2,\dots, i_m) < (k_1,k_2,\dots,k_m)$ (e.g. $i_j \leqslant k_j$ for any $j$ and $i_{j_0} < k_{j_0}$ for some $j_0$) we have that
$$x_{1}^{i_1}x_{2}^{i_2}\dots x_{m}^{i_m}\partial_{x_1}^{i_1} \dots \partial_{x_m}^{i_m}FG(\mathbf{ x}) \ll  Z_FZ_G(\mathbf{ x}) (F_{1}+G_{1})(\mathbf{ x})^{i_1}\dots (F_{m}+G_{m})(\mathbf{ x})^{i_m}$$
holds for any $F,G$ satisfying the assumptions as before.

Without loss of generality, assume that $k_1>0$, such that
\begin{align*}
x_{1}^{k_1}x_{2}^{k_2}\dots x_{m}^{k_m}\partial_{x_1}^{k_1} \dots \partial_{x_m}^{k_m} FG(\mathbf{ x}) = & x_{1}^{k_1}x_{2}^{k_2}\dots x_{m}^{k_m}\partial_{1}^{k_1-1} \dots \partial_{m}^{k_m} \left[G(\mathbf{ x}) \partial_{x_1}F(\mathbf{ x})+  F(\mathbf{ x})\partial_{1}G(\mathbf{ x})\right].
\end{align*}
Then, by Lemma \ref{lem1} and induction, we complete the proof.
 \end{proof}

\begin{lem} \label{lem3}
Let $F(\mathbf{ y})$ be a $\mathbb{R}^n$ to $\mathbb{R}$ map, with type $(Z_F: F_1, F_2, \dots ,F_n)$. Let $\mathbf{ G}(\mathbf{ x}) = (G_1(\mathbf{x}), \dots, G_n(\mathbf{x}))$ be a $\mathbb{R}^m$ to $\mathbb{R}^n$ map with each $G_k(\mathbf{x}) \neq 0$ for any $\mathbf{x}$. Moreover, assume that $G_k(\mathbf{x})$ has $\mathbf{x}$-type $(Z_{G_k}: G_{k1},\dots, G_{km})$ with each $G_{kj}(\mathbf{x})$ nonnegative. 

Then $F(\mathbf{ G}(\mathbf{ x}))$ is a  $\mathbb{R}^m$ to $\mathbb{R}$ map, with $\mathbf{x}$-type $(Z_{F(G)}:F(G)_1,\dots,F(G)_m)$, where $Z_{F(G)} = Z_F(\mathbf{ G}(\mathbf{ x}))$ and $$F(G)_j = \sum_{k}\frac{\left[F_k(\mathbf{ G}(\mathbf{ x}))Z_{G_k}(\mathbf{ x})+G_k(\mathbf{ x})\right]G_{kj}(\mathbf{ x})}{G_k(\mathbf{ x})}.$$
\end{lem}

\begin{proof}
Induction on $(i_1,i_2,\dots, i_m)$, assume that for any $(i_1,i_2,\dots, i_m) < (k_1,k_2,\dots,k_m)$ (e.g. $i_j \leqslant k_j$ for any $j$ and $i_{j_0} < k_{j_0}$ for some $j_0$), we have that
$$x_{1}^{i_1}x_{2}^{i_2}\dots x_{m}^{i_m}\partial_{x_1}^{i_1} \dots \partial_{x_m}^{i_m}F(G(\mathbf{ x})) \ll Z_F(\mathbf{ G}(\mathbf{ x}))\prod_j\(\sum_{k}\frac{\left[F_k(\mathbf{ G}(\mathbf{ x}))Z_{G_k}(\mathbf{ x})+G_k(\mathbf{ x})\right]G_{kj}(\mathbf{ x})}{G_k(\mathbf{ x})}\)^{i_j}$$
holds for any $F, G_k$ satisfying the assumptions as above.

Without loss of generality, one can assume that $k_1>0$. Then
\begin{align*}
x_{1}^{k_1}x_{2}^{k_2}\dots x_{m}^{k_m}\partial_{x_1}^{k_1} \dots \partial_{x_m}^{k_m} F(G(\mathbf{ x})) = & x_{1}^{k_1}x_{2}^{k_2}\dots x_{m}^{k_m}\partial_{x_1}^{k_1-1} \dots \partial_{x_m}^{k_m} \left[\partial_{x_1}F(G(\mathbf{ x}))\right].\\
= & x_{1}^{k_1}x_{2}^{k_2}\dots x_{m}^{k_m}\partial_{x_1}^{k_1-1} \dots \partial_{x_m}^{k_m} \left[\sum_j(\partial_{j}F)(G(\mathbf{ x}))\partial_{x_1}G_j(\mathbf{ x})\right].\\
\ll & Z_F(\mathbf{ G}(\mathbf{ x}))\prod_j\(\sum_{k}\frac{\left[F_k(\mathbf{ G}(\mathbf{ x}))Z_{G_k}(\mathbf{ x})+G_k(\mathbf{ x})\right]G_{kj}(\mathbf{ x})}{G_k(\mathbf{ x})}\)^{i_j}
\end{align*}
by Lemmas \ref{lem1}, \ref{lem2} and induction.
 \end{proof}
 
Now, we are ready to study the weight functions $I_{L,X,D}(x,y,d)$, $H^{\pm,\pm}(m,n,h,d,q)$ and $u^{\pm}(v,h,q;d,x)$.\\ 

After changing variables in \eqref{I},
\begin{align*}
I^{\pm}_{L,X,D}(x,y,d) 
=  2 \eta_L^{\pm} \sqrt{\frac{Xx}{Ld^2}}h_0\(\frac{d}{D}\)h\(\frac{y}{X}\)\int_{0}^{\infty}\frac{h\(u\)}{2\sqrt{u}}J_{f}\(\sqrt{\frac{Xxu}{Ld^2}}\)J_{\kappa-1}\( 4 \pi \sqrt{\frac{Xyu}{d^2}}\)du.
\end{align*}

By Lemmas \ref{lem7}, \ref{lem2}, \ref{lem3} and changing variables back and forth, one has
\begin{lem} \label{lem9}
$I^{\pm}_{L,X,D}(x,y,d)$ has $(x, y, d)$-type
$$\(\min\left\{1,\(\frac{xX}{LD^2}\)^{\frac{1}{2}-2r_f}\right\}: \sqrt{\frac{xX}{LD^2}} + 1, \sqrt{\frac{yX}{D^2}} + 1+Z_h, 1+Z_h\),$$
and is support on $\mathbb{R}^{+}\times [1/2X,5/2X]\times[1/2D,5/2D]$. Moreover, it also satisfies
$$I_{L,X,D}(x,y,d)\ll_{m}\(1+Z_h\)^m\left(\sqrt{\frac{xX}{LD^2}}-\sqrt{\frac{yX}{D^2}}\right)^{-m}.$$
All the above implied constants may depend on $t_f$ and $\kappa$.
\end{lem}

By changing variables in \eqref{H_L}, one has
\begin{align}
\label{HLL}
H^{\pm_1,\pm_2}(m,n,h,d,q) :=\iint_{0}^{\infty}F\(\frac{q^2N\xi^2}{m},\frac{q^2N\mu^2}{n},hc_0,d\)\mathcal{J}_f^{\pm_1}\(\xi\)\mathcal{J}_g^{\pm_2}\(\mu\)d\xi d\mu,
\end{align}
where 
$$F(x,y,hc_0,d)= 4 \eta_g^{\pm_1}\eta_f^{\pm_2} I\(x,y,d\)w_{\delta}\(\Delta\(x,y,hc_0\)\),$$
 $\Delta(x,y,hc_0) = l_1x-l_2y -hc_0$, and $w_{\delta}\(\Delta\) = \frac{1}{2\delta}\int_{-\delta}^{\delta}e\(\Delta x\)dx$.\\

\begin{lem} \label{lem10}
For any given integers $m,n \geqslant 0$, $H^{\pm_1,\pm_2}(x,y,h,d,q)$ has $(x, y, h, d,q)$-type
$$\(B_H(m,n): Z_1(S_1) +1, Z_2(S_2) + 1,1 , 1+Z_3, Z_1(S_1)+Z_2(S_2)+ 1\),$$
where 
$$B_H(m,n):=\min\left\{1,\(\frac{xS_1}{q^2N}\)^{\frac{1}{2}-2r_f}\right\}\min\left\{1,\(\frac{yS_2}{q^2N}\)^{\frac{1}{2}-2r_g}\right\}\left(\frac{Z_1(S_1) +1}{\sqrt{xS_1/q^2N}}\right)^m\left(\frac{Z_2(S_2) + 1}{\sqrt{yS_2/q^2N}}\right)^n$$
and $r_f, r_g$ are defined in Remark \ref{General_J_Bessel}. ( When $f$ ( resp. $g$ ) is Maa{\ss} form with $t_f = 0$ ( resp. $t_g=0$), we will have $\frac{1}{2}-\epsilon$ power instead of power $\frac{1}{2}$. )
\end{lem}
 
\begin{proof}
In order to establish the $(x, y, h, d,q)$-type of 
$H^{\pm_1,\pm_2}(x,y,h,d,q)$, we need to bound 
\begin{align}
\label{pHL}
&\partial_{x}^{i_x}\partial_{y}^{i_y}\partial_{h}^{i_h}\partial_{d}^{i_d}\partial_{q}^{i_q}H^{\pm_1,\pm_2}(x,y,h,d,q) \\
& = \iint_{0}^{\infty}\partial_{x}^{i_x}\partial_{y}^{i_y}\partial_{h}^{i_h}\partial_{d}^{i_d}\partial_{q}^{i_q}F\(\frac{q^2N\xi^2}{x},\frac{q^2N\mu^2}{y},h,d\)\mathcal{J}_f^{\pm_1}\(\xi\)\mathcal{J}_g^{\pm_2}\(\mu\)d\xi d\mu.
\end{align}
First of all, we establish the type of the integrand.

Set $\mathbf{a} = (x, y,h,d,q)$. Let $\mathbf{ I} = (i_x, i_y, i_h, i_d, i_q)$ be a vector of nonnegative integers. Let $\mathbf{a}^{\mathbf{ I}} : = x^{i_x}y^{i_y}h^{i_h}d^{i_d}q^{i_q}$, $\partial_{\mathbf{a}}^{\mathbf{ I}} : =\partial_{x}^{i_x}\partial_{y}^{i_y}\partial_{h}^{i_h}\partial_{d}^{i_d}\partial_{q}^{i_q}$.

Let $$K_{\mathbf{ I}}(\xi,\mu) := \partial_{\mathbf{a}}^{\mathbf{ I}}F\(\frac{q^2N\xi^2}{x},\frac{q^2N\mu^2}{y},h,d\).$$

Let $\mathbf{b} = (\xi, \mu, x,y,h,d,q)$ be a vector, and $\mathbf{G}(\mathbf{b}) :=(G_1(\mathbf{b}), \dots, G_4(\mathbf{b}))= \(\frac{q^2N\xi^2}{x},\frac{q^2N\mu^2}{y},h,d\)$ be a $\mathbb{R}^7$-to-$\mathbb{R}^4$ map.

It is easy to check that 
$G_1(\mathbf{b})$ has $\mathbf{b}$-type $(G_1(\mathbf{b}):1, 0, 1, 0, 0, 0, 1)$. $G_2(\mathbf{b})$ has $\mathbf{b}$-type $(G_2(\mathbf{b}):0, 1, 0, 1, 0, 0, 1)$. $G_3(\mathbf{b})$ has $\mathbf{b}$-type $(G_3(\mathbf{b}):0,0,0,0,1, 0, 0)$. $G_4(\mathbf{b})$ has $\mathbf{b}$-type $(G_4(\mathbf{b}):0, 0, 0, 0, 0, 1, 0)$.

By \eqref{HLL}, $F(x,y,hc_0,d)=4 \eta_g^{\pm_1}\eta_f^{\pm_2} I\(x,y,d\)w_{\delta}\(\Delta\(x,y,hc_0\)\)$. Also note that $w_{\delta}\(\Delta\(x,y,hc_0\)\)$ has $(x,y,h)$-type $(1: \delta l_1x, \delta l_2y,\delta hc_0)$ and $I\(x,y,d\)$ has $(x,y,d)$-type $(1:Z_1,Z_2,Z_3)$. Hence by Lemma \ref{lem2}, \ref{lem9}, $F(x,y,hc_0,d)$ has $(x,y,h,d)$-type
$$\( 1 : Z_1 + \delta l_1x + 1,Z_2 + \delta l_2y+1, \delta hc_0, Z_3+1\).$$

Set 
$$F_x := Z_1(t_x)+\delta l_1t_x +1,\,  \ F_y := Z_2(t_y)+\delta l_2t_y +1,\,  \ F_h:=\delta hc_0+1,\,  \ F_d :=Z_3+1,$$ 
$$F_q :=  Z_1(t_x)+Z_2(t_y)+\delta l_1t_x +\delta l_2t_y +1, $$
where $t_x := \frac{q^2N\xi^2}{x}$, $t_y := \frac{q^2N\mu^2}{y}$.
Also set $$\mathbf{F}^{\mathbf{I}} := F_x^{i_x}F_y^{i_y}F_{h}^{i_h}F_d^{i_d}F_q^{i_q}.$$

Then, by Lemmas \ref{lem3} and \ref{lem1}, one obtains that $K_{\mathbf{ I}}(\xi,\mu)=\partial_{\mathbf{a}}^{\mathbf{I}}F(\mathbf{G}(\mathbf{b}))$ has $(\xi,\mu)$-type
$$\(\mathbf{F}^{\mathbf{I}}: F_x, F_y\).$$

Next, we study the integral in \eqref{pHL}. We will use a similar method as the one in the proof of Lemma \ref{lem7} to establish the bound for $m=1, n=0$. Note that $K_{\mathbf{ I}}(\xi,\mu)=0$ if $(t_x,t_y) \notin [1/2S_1,5/2S_1]\times [1/2S_2, 5/2S_2]$.

We proceed by considering the integral over those $(\xi, \mu)$ where $K_{\mathbf{ I}}(\xi,\mu) \neq 0$. Different techniques will be applied when $\xi \geqslant 1$ and when $0< \xi \ll 1$. 

Let $\xi \sim \sqrt{xS_1/q^2N}$ and $\mu \sim \sqrt{yS_2/q^2N}$, or $K_{\mathbf{ I}}(\xi,\mu) = 0$. We first treat the case that $\sqrt{xS_1/q^2N} \gg 1$. When $f$ is holomorphic, by \eqref{JB}, the right hand side of \eqref{pHL} becomes
\begin{align}
\iint_{0}^{\infty}K_{\mathbf{ I}}(\xi,\mu)\(e^{i\xi}W_{2t_f}(\xi)+e^{-i\xi}\overline{W}_{2t_f}(\xi)\)\mathcal{J}_g^{\pm_2}\(\mu\)d\xi d\mu.
\end{align} 
We consider $\iint_{0}^{\infty}K_{\mathbf{ I}}(\xi,\mu)e^{i\xi}W_{2t_f}(\xi)\mathcal{J}_g^{\pm_2}\(\mu\)d\xi d\mu$. The other term can be treated similarly. From \eqref{JB}, \eqref{Wk}, Remark \ref{General_J_Bessel} and integration by parts, one gets the upper bound as 
\begin{align}
\ll \min\left\{1,\(\frac{yS_2}{q^2N}\)^{\frac{1}{2}-2r_g}\right\}\left(\frac{Z_1\(S_1\) + \delta l_1S_1+1}{\sqrt{xS_1/q^2N}}\right)\mathbf{F}^{\mathbf{I}}.
\end{align} 

We then consider the case that $\sqrt{xS_1/q^2N} \ll 1$. Then by Remark \ref{General_J_Bessel}, we can bound the integral trivially as 
\begin{align}
\ll \min\left\{1,\(\frac{xS_1}{q^2N}\)^{\frac{1}{2}-2r_f}\right\}\min\left\{1,\(\frac{yS_2}{q^2N}\)^{\frac{1}{2}-2r_g}\right\}\mathbf{F}^{\mathbf{I}} 
\end{align}  

For a Maa{\ss} form $f$, we can still do integration by parts for $\mathcal{J}_f^{+}$, but using trivial bound for $\mathcal{J}_f^{-}(y)$ which is exponential decay for large values of $y$. By repeating this process and noticing that $\delta$ is negligible, we obtain the desired bound.
\end{proof}


The following lemma is a consequence of Lemmas \ref{lem1} and \ref{lem10}.
\begin{lem} \label{lem11}
Let $\widetilde{H}^{+,+}_{A,B}(m,n,h,d,q) := H^{+,+}(m,n,h,d,q)h_0\(\frac{m}{A}\)h_0\(\frac{n}{B}\)$. Set $h^{\pm}(v)$ be functions such that $h^{\pm}(v)=1$ when $\pm v\geqslant 2/3$ and $h^{\pm}(v)=0$ when $\pm v \leqslant 1/3$. Let 
$$ u^+(v,h,q;d,x) : = \frac{d}{d x}\widetilde{H}_{A,B}(\frac{v+l_1x}{l_2},x,h,d,q)h^{+}(v),$$
$$u^-(v,h,q; d,x) : = \frac{d}{d x}\widetilde{H}_{A,B}(x,\frac{l_2x-v}{l_1},h,d,q)h^{-}(v).$$

Then $u^{+}(v, h , q; d, x) = 0$ when $20Al_2 <  l_1B$. It is supported on the region such that $ v \in [0, \frac{5}{2}Al_2] $ and $x \in [\frac{1}{2}B, \frac{5}{2}B]$. Moreover, it has $(v, h, q, d)$-type
$$\(B_{u^{+}} : Z_1+1, 1, Z_1+Z_2+1, Z_3+1\),$$
where $B_{u^{+}}= \(\frac{l_1Z_1}{Al_2}+\frac{Z_2}{B}\)B_H(0,0)$.

Similarly, $u^{-}(v, h ,q; d, x) = 0$ when $Al_2 >  20Bl_1$. It is supported on the region such that $ -v \in [0,\frac{5}{2}Bl_1] $ and $x \in [\frac{1}{2}A, \frac{5}{2}A]$. Moreover, it has $(v, h, q, d)$-type
$$\(B_{u^{-}} : Z_2+1, 1, Z_1+Z_2+1, Z_3+1\),$$
where $B_{u^{-}} = \(\frac{Z_1}{A}+\frac{l_2Z_2}{Bl_1}\)B_H(0,0)$.\\
\end{lem}

\section{The Proof of Theorem \ref{pro1}}\label{seLS}
In this section, we prove the large sieve inequality Theorem \ref{pro1}. 
Our arguments are motivated by ideas demonstrated in \cite{B1}, \cite{DI} and \cite{P1}.

\begin{proof}[Proof of Theorem \ref{pro1}]
Consider the case of $\mathfrak{S}_{+}$. The treatment of $\mathfrak{S}_{-}$ is similar.\\

As in \cite{DI}, we consider the Fourier transform of $u$ as
\begin{align*}
G(t_1, t_2,t_3; x) = \iiint_{\mathbb{R}^3} u\left(x_1,x_2,\frac{4\pi\sqrt{x_1x_2w}}{s\sqrt{r}x},x_3\right)e(-t_1x_1-t_2x_2-t_3x_3)dx_1dx_2dx_3,
\end{align*}
so that
$$u\left(x_1,x_2,\frac{4\pi\sqrt{x_1x_2w}}{s\sqrt{r}x},x_3\right) = \iiint_{\mathbb{R}^3} G(t_1,t_2,t_3; x)e(t_1x_1+t_2x_2+t_3x_3)dt_1dt_2dt_3.$$

Furthermore, one obtains
\begin{align*}
& \frac{\partial^p}{\partial x^p}G(t_1, t_2,t_3; x)\\
 = & (2\pi it_1)^{p_1}(2\pi it_2)^{p_2}(2\pi it_3)^{p_3}\iiint_{\mathbb{R}^3} \frac{\partial^{p_1+p_2+p_3+p}}{\partial x_1^{p_1}\partial x_2^{p_2}\partial x_3^{p_3}\partial x^p}u\left(x_1,x_2,\frac{4\pi\sqrt{x_1x_2w}}{s\sqrt{r}x},x_3\right)\\
 &\times e(-t_1x_1-t_2x_2-t_3x_3)dx_1dx_2dx_3 \\
 \ll & (t_1V/Z)^{-p_1}(t_2H/Z)^{-p_2}(t_3D/Z)^{-p_3}(\sqrt{VHw}Z/s\sqrt{r}Q)^{-p}VH
\end{align*}
by integration by parts. Also note that $G(t_1, t_2,t_3; x)$ is compactly supported in terms of $x$.

Thus,
\begin{align}
& \sum_{\substack{q\\(q,r)=1}}\sum_{h,d}\sum_{v}\frac{1}{q}S(v\overline{r},  hw, sq)a(v)b(h,d)u(v,h,q,d) \label{11}\\
= & \iiint_{\mathbb{R}^3}\sum_{\substack{q\\(q,r)=1}}\sum_{h,d}\sum_{v}\frac{1}{q}S(v\overline{r},  hw, sq)a(v)b(h,d)G\(t_1,t_2,t_3; \frac{4\pi\sqrt{vhw}}{q}\)\\
& \times e(t_1v+t_2h+t_3d)dt_1dt_2dt_3.
\end{align}

We now study the integrand 
\begin{equation}\label{integrant}
\sum_{\substack{q\\(q,r)=1}}\sum_{h,d}\sum_{v}\frac{1}{q}S(v\overline{r},  hw, sq)a(v)b(h,d)G\(t_1,t_2,t_3; \frac{4\pi\sqrt{vhw}}{q}\)e(t_1v+t_2h+t_3d),
\end{equation}
saving integration over $t_1, t_2, t_3$ for later.

Let $\Xi=\sqrt{VHw}/s\sqrt{r}Q$, and for fixed $p_1, p_2, p_3$, define
$$\varphi\(\frac{4\pi\sqrt{vhw}}{s\sqrt{r}q}\) := (DVH)^{-1}(t_1V/Z)^{p_1}(t_2H/Z)^{p_2}(t_3D/Z)^{p_3}G\(t_1, t_2,t_3; \frac{4\pi\sqrt{vhw}}{s\sqrt{r}q}\).$$
So that $\varphi(x)$ is supported on $[\Omega^{-1}\Xi, \Omega \Xi]$ for some absolute positive number $\Omega$. Moreover, 
$$\varphi^{(i)}(x)\ll (Z/\Xi)^i.$$

we absorb the $e(t_1 v+ t_2 h + t_3 d)$ factor into $a(v)$, $b(h,d)$ and pull out a factor 
$e\left(v \frac{\bar{s}}{r}\right)$ from $a(v)$. Since these all have norm $1$, this will not affect our final bound. Therefore, it suffices to bound 
\begin{equation}\label{integrant_2}
\sum_{h,d}\sum_{v}a(v)e\(v\frac{\bar{s}}{r}\)b(h,d)\sum_{\substack{q>0 \\ (q,r)=1}} \frac{1}{s\sqrt{r}q} S(v\overline{r}, hw ; sq) \varphi(\frac{4\pi \sqrt{vhw}}{s\sqrt{r}q})
\end{equation}

In the trace formula (Lemma \ref{lemKT}), we consider the case of $\Gamma =\Gamma_0(rs)$ with cusps $\mathfrak{a}$, $\mathfrak{b}$ such that $\mathfrak{a} = 1/s$, and $\mathfrak{b} = \infty \sim 1/rs$. Then, as defined in Lemma \ref{lem10}, we have that $\mu(\mathfrak{b}) = (1,rs)/rs = 1/rs$ and $\mu(\mathfrak{a}) = (s,r)/rs = 1/rs$. 

Then, by \cite{DI} (1.6), we can rewrite \eqref{integrant_2} as
\begin{align}
\mathcal{S}:=\sum_{h,d}\sum_{v}a(v)b(h,d)\sum_{\gamma}^{\Gamma} \frac{1}{\gamma} S_{\mathfrak{ab}}(v, hw ; \gamma) \varphi\(\frac{4\pi \sqrt{vhw}}{\gamma}\).
\end{align}

Applying Lemma \ref{lemKT}, the sum above equals
\begin{align}
& \frac{1}{2\pi}\sum_{k = 0(2)}\sum_{j}\frac{i^k(k-1)!}{(4\pi)^{k-1}}\sum_{v,h,d}a(v)b(h,d)\overline{\psi}_{jk}(\mathfrak{a},v)\psi_{jk}(\mathfrak{b},hw)\tilde{\varphi}(k-1)\\
& + \sum_{j \geqslant 1} \sum_{v,h,d}a(v)b(h,d)\frac{\overline{\rho}_{j\mathfrak{a}}(v)\rho_{j\mathfrak{b}}(hw)}{\cosh (\pi t_j)}\hat{\varphi}(t_j)\\
& + \frac{1}{\pi}\sum_{j}\int_{-\infty}^{\infty}\sum_{v,h,d}a(v)b(h,d)\left(\frac{v}{hw}\right)^{-it}\overline{\varphi}_{j\mathfrak{a}}\left(v,\frac{1}{2}+it\right)\varphi_{j\mathfrak{b}}\left(hw,\frac{1}{2}+it\right)\hat{\varphi}(t)dt.
\end{align}

Applying Cauchy's inequality, we get that
\begin{align}
& \mathcal{S}^2 \\
 \ll & \sum_{j \geqslant 1}\left|\frac{\hat{\varphi}(t_j)}{\cosh (\pi t_j)}\right|\left|\sum_{v}a(v)\rho_{j\mathfrak{a}}(v)\right|^2\sum_{j \geqslant 1}\left|\frac{\hat{\varphi}(t_j)}{\cosh (\pi t_j)}\right|\left|\sum_{h,d} b(h,d)\overline{\rho}_{j\mathfrak{b}}(hw)\right|^2\\
& + \sum_{k = 0(2)}|\tilde{\varphi}(k-1)|\frac{(k-1)!}{(4\pi)^{k-1}}\sum_{j}\left|\sum_{v}a(v)\overline{\psi}_{jk}(\mathfrak{a},v)\right|^2
 \sum_{k = 0(2)}|\tilde{\varphi}(k-1)|\frac{(k-1)!}{(4\pi)^{k-1}}\sum_{j}\left|\sum_{h,d}b(h,d)\overline{\psi}_{jk}(\mathfrak{b},hw)\right|^2\\
& + \sum_{j}\int_{-\infty}^{\infty}|\hat{\varphi}(t)|\left|\sum_v a(v)\left(v\right)^{-it}\overline{\varphi}_{j\mathfrak{a}}\left(v,\frac{1}{2}+it\right)\right|^2dt \sum_{j}\int_{-\infty}^{\infty}|\hat{\varphi}(t)|\left|\sum_{h,d}b(h,d)(hw)^{it}\varphi_{j\mathfrak{b}}\left(hw,\frac{1}{2}+it\right)\right|^2dt.
\end{align}

We will bound
$$\sum_{j \geqslant 1}\left|\frac{\hat{\varphi}(t_j)}{\cosh (\pi t_j)}\right|\left|\sum_{v}a(v)\rho_{j\mathfrak{a}}(v)\right|^2 \text{ , }\sum_{j \geqslant 1}\left|\frac{\hat{\varphi}(t_j)}{\cosh (\pi t_j)}\right|\left|\sum_{h,d}b(h,d)\overline{\rho}_{j\mathfrak{b}}(hw)\right|^2.$$
Consider the sum over $v$ first. Set $T=Z+2\Xi+1$. We have
\begin{align}
& \sum_{j \geqslant 1}\left|\frac{\hat{\varphi}(t_j)}{\cosh (\pi t_j)}\right|\left|\sum_{v}a(v)\rho_{j\mathfrak{a}}(v)\right|^2 \\
= & 
\sum_{|t_j| \leqslant T} \left|\frac{\hat{\varphi}(t_j)}{\cosh (\pi t_j)}\right|\left|\sum_{v}a(v)\rho_{j\mathfrak{a}}(v)\right|^2 + \sum_{|t_j| > T} \left|\frac{\hat{\varphi}(t_j)}{\cosh (\pi t_j)}\right|\left|\sum_{v}a(v)\rho_{j\mathfrak{a}}(v)\right|^2.
\end{align}
From Lemma \ref{lemB} and \ref{lem13}, we have 
$$\sum_{|t_j| < T}\left|\frac{\hat{\varphi}(i\kappa_j)}{\cosh (\pi t_j)}\right|\left|\sum_{v}a(v)\rho_{j\mathfrak{a}}(v)\right|^2 \ll
\left(\frac{1 + \log\(\frac{\Xi}{Z}\)+ \left(\frac{\Xi}{Z}\right)^{-2\theta}}{{1+\frac{\Xi}{Z}}}\right)\left(T^2+ \frac{V^{1+\varepsilon}}{rs}\right)\|a(v)\|^2. $$
For the second term, we split the sum dyadically to obtain that
\begin{align}
\sum_{|t_j| > T}\left|\frac{\hat{\varphi}(i\kappa_j)}{\cosh (\pi t_j)}\right|\left|\sum_{v}a(v)\rho_{j\mathfrak{a}}(v)\right|^2 \ll & \sum_{i>0}\left(\frac{1}{2}\right)^{2i}\left(\frac{Z}{T}\right)^{4}\left(\frac{1}{T^{1/2}}+\frac{\Xi(1+\log T)}{T}\right) \left(T^2+ \frac{V^{1+\varepsilon}}{rs}\right)\|a(v)\|^2\\
\ll & \left(\frac{Z}{T}\right)^{4}\left(\frac{1}{T^{1/2}}+\frac{\Xi(1+\log T)}{T}\right) \left(T^2+ \frac{V^{1+\varepsilon}}{rs}\right)\|a(v)\|^2.
\end{align}
Therefore, by noticing that $Z\geqslant 1$, we have the bound
\begin{align}
\label{large sieve 1}
& \sum_{j \geqslant 1}\left|\frac{\hat{\varphi}(t_j)}{\cosh (\pi t_j)}\right|\left|\sum_{v}a(v)\rho_{j\mathfrak{a}}(v)\right|^2
 \ll  
 \left(\frac{1 + \left(\frac{\Xi}{Z}\right)^{-2\theta}}{{Z+X}}\right) \left(Z^2+\Xi^2+ \frac{V^{1+\varepsilon}}{rs}\right)\|a(v)\|^2Z(\Xi Z)^{\epsilon}.
\end{align}

Next, we estimate the sum over $n$. For this term, we need the arithmetic property of Hecke-eigenvalues
\begin{align}
\rho_{j\infty}(hw) = \sum_{u|(h,w)}\rho_{j\infty}\(\frac{h}{u}\)\lambda_{j}\(\frac{w}{u}\)\label{51}
\end{align}
when $(w,rs) = 1$.

Thus, one can estimate this sum via a similar way as 
\begin{align}
 &\sum_{j \geqslant 1}\left|\frac{\hat{\varphi}(t_j)}{\cosh (\pi t_j)}\right|\left|\sum_{h,d}b(h,d)\overline{\rho}_{j\mathfrak{b}}(hw)\right|^2\label{52}
 =  \sum_{j \geqslant 1}\left|\frac{\hat{\varphi}(t_j)}{\cosh (\pi t_j)}\right|\left|\sum_{h,d}b(h,d)\sum_{u|(h,w)}\overline{\rho}_{j\mathfrak{b}}\(\frac{h}{u}\)\overline{\lambda}_{j}\(\frac{w}{u}\)\right|^2\notag\\
\ll & \sum_{j \geqslant 1}\left|\frac{\hat{\varphi}(t_j)}{\cosh (\pi t_j)}\right|\sum_{u|w}\left|\overline{\lambda}_{j}\(\frac{w}{u}\)\right|^2\sum_{u|w}\left|\sum_{h \equiv 0(\text{ mod }u)}\sum_db(h,d)\overline{\rho}_{j\mathfrak{b}}\(\frac{h}{u}\)\right|^2\notag\\
\ll &\left(\frac{1 + \left(\frac{\Xi}{Z}\right)^{-2\theta}}{{Z+\Xi}}\right) \left(Z^2+\Xi^2+ \frac{H^{1+\varepsilon}}{rs}\right)w^{2\theta}\|B(h)\|^2Z(\Xi Zw)^{\epsilon}.
\end{align}
For the last step, we used \eqref{large sieve 1} and the Kim-Sarnak bound $|\lambda(w)| \leqslant \tau(w)w^{\theta}$ (see \cite{K}).


The holomorphic case is similar. In order to treat the case of Eisenstein series, we need the Heck eigenvectors of space generated by incomplete Eisenstein series. As in Remark \ref{Eis}, we refer to the notation in \cite{GJ}, where we have a basis of Eisenstein series indexed by a finite set
$$\{(t, \chi_1,\chi_2,b) | t\in \mathbb{R}, \chi_1\chi_2 = 1, b \in \mathcal{B}(\chi_1,\chi_2)\}.$$

Furthermore, as explained in section 2.6-2.8 \cite{BH} and in \cite{BHM1} , the Hecke multiplicativity of coefficients of Eisenstein series and the large sieve inequality also hold. As a consequence, a similar bound holds true for this part as well. We can also bound the continuous part via a direct computation (see \cite{BHM} Section 3.2).

Thus, one can get the final bound of \eqref{integrant_2} as
\begin{align}
 \left(\frac{1+\left(\frac{\Xi}{Z}\right)^{-2\theta}}{Z+\Xi}\right)\left(Z+\Xi+ \frac{V^{1/2+\varepsilon}}{(rs)^{1/2}}\right)\left(Z+\Xi + \frac{H^{1/2+\varepsilon}}{(rs)^{1/2}}\right)w^{\theta}\|a(v)\|_2\|B(h)\|_2Z^2(w\Xi Z)^{\varepsilon}. \label{fLS}
\end{align}

Now assume that
$$B(\Xi,Z,V,H,Q) =  \left(\frac{1+\left(\frac{\Xi}{Z}\right)^{-2\theta}}{Z+\Xi}\right)\left(Z+\Xi+ \frac{V^{1/2+\varepsilon}}{(rs)^{1/2}}\right)\left(Z+\Xi + \frac{H^{1/2+\varepsilon}}{(rs)^{1/2}}\right)w^{\theta}(w\Xi Z)^{\varepsilon},$$

For fixed $t_1,t_2,t_3$, using the bound \eqref{fLS}, \eqref{11} is bounded above by
\begin{align*}
& \iiint DVH(t_1V/Z)^{-p_1}(t_2H/Z)^{-p_2}(t_3D/Z)^{-p_3}B(\Xi,Z,V,H,Q)Z^2\\
&\times\left|\sum_v|a(v)|^2\right|^{1/2}\left|\sum_h|\sum_db(h,d)e(t_3d)|^2\right|^{1/2} dt_1dt_2dt_3\\
\ll_{\epsilon} & B(\Xi,Z,V,H,Q)(1+Z^8)\|a(v)\|\|B(h)\|.
\end{align*}
In the last step, we chose $p_1 = p_2 =p_3=0$ when $t_1<ZV^{-1-\epsilon}$, $t_2<ZH^{-1-\epsilon}$, $t_3<ZD^{-1-\epsilon}$ and $p_1 = p_2 = p_3=2$ otherwise. Therefore, we complete the proof.
\end{proof}

\end{document}